\documentclass{article}
\usepackage{amsmath,amssymb}
\usepackage{enumerate}
\usepackage{dsfont}

\allowdisplaybreaks[1]

\newtheorem{theorem}{Theorem}[section]
\newtheorem{lemma}[theorem]{Lemma}
\newtheorem{proposition}[theorem]{Proposition}

\newtheorem{definition}[theorem]{Definition}
\newtheorem{remark}[theorem]{Remark}
\newtheorem{example}[theorem]{Example}

\newcommand{\Z}{\mathbb{Z}}
\newcommand{\F}{\mathbb{F}}

\newcommand{\im}{\operatorname{Im}}

\newcommand{\id}{\operatorname{id}}

\newcommand{\sym}{\operatorname{Sym}}
\newcommand{\aut}{\operatorname{Aut}}
\newcommand{\End}{\operatorname{End}}

\newcommand{\soc}{\operatorname{Soc}}

\newcommand{\Aut}{\operatorname{Aut}}

\newenvironment{proof}{\par\noindent{\bf Proof.}}{$\qed$\par\bigskip}
\newcommand{\qed}{\enspace\vrule  height6pt  width4pt  depth2pt}
\usepackage{color}

\begin{document}

\title{Asymmetric product of left braces and simplicity; new solutions of the Yang-Baxter equation}
\author{D. Bachiller \and F. Ced\'o \and E. Jespers \and J. Okni\'{n}ski}
\date{}

\maketitle

\begin{abstract}
The problem of constructing all the non-degenerate involutive
set-theoretic solutions of the Yang-Baxter equation  recently has
been reduced to the problem of describing all the left braces. In
particular, the classification of all  finite left braces is
fundamental in order to describe all finite such solutions of the
Yang-Baxter equation. In this paper we continue the study  of finite
simple left braces with the emphasis on the application of the
asymmetric product of left braces in order to construct new classes
of simple left braces. We do not only construct new classes but also
we interpret all previously known constructions  as asymmetric
products. Moreover, a construction is given of finite simple left
braces with a multiplicative group that is solvable of arbitrary
derived length.
\end{abstract}



\section{Introduction}
Rump \cite{Rump} introduced braces to study a class of solutions of
the Yang-Baxter equation, an important equation in mathematical
physics (see \cite{Yang,K}). Recall that a set-theoretic solution of
the Yang-Baxter equation is a map $r\colon X\times X\longrightarrow
X\times X$, such that $X$ is a non-empty set and
$$r_{12}r_{23}r_{12}=r_{23}r_{12}r_{23},$$
where $r_{ij}$ denotes the map $X\times X\times X\longrightarrow
X\times X\times X$ acting as $r$ on the $(i,j)$ components and as
the identity on the remaining component. The solution $r$ is
involutive if $r^2=\id_{X^2}$, and it is non-degenerate if
$r(x,y)=(f_x(y),g_y(x))$ and $f_x$ and $g_y$ are bijective maps from
$X$ to $X$ , for all $x,y\in X$. A systematic study of
non-degenerate, involutive set-theoretic solutions of the
Yang-Baxter equation (solutions of the YBE for short) was initiated
in \cite{ESS} and \cite{GIVdB}. Solutions of this type have received
a lot of attention in recent years (see for example
\cite{BCJ,BCJO1,CGIS,CJO,CJO2,GI,GIC,gat-maj,GI-braces,Rump1,Rump,Smokt,Smokt2,ven2016}).

Every solution $(X,r)$ of the YBE has associated two important
groups: its structure group $G(X,r)$  defined by the presentation
$\langle x\in X\mid xy=f_x(y)g_y(x)\rangle$ (where
$r(x,y)=(f_x(y),g_y(x))$), and its permutation group
$\mathcal{G}(X,r)=\langle f_x\mid x\in X\rangle$, a subgroup of
$\sym_X$. There exists a homomorphism $\phi\colon
G(X,r)\longrightarrow \mathcal{G}(X,r)$ such that $\phi(x)=f_x$, for
all $x\in X$. Etingof, Schedler and Soloviev proved in \cite{ESS}
that if $(X,r)$ is a finite solution of the YBE, then $G(X,r)$ and
$\mathcal{G}(X,r)$ are solvable groups.

Rump introduced braces to study this type of solutions \cite {Rump},
and noted that for every solution $(X,r)$ of the YBE, the groups
$G(X,r)$ and $\mathcal{G}(X,r)$ have a natural structure of left
braces such that $\phi\colon G(X,r)\longrightarrow \mathcal{G}(X,r)$
becomes a homomorphism of left braces. A left brace is a set $B$
with two operations, $+$ and $\cdot$, such that $(B,+)$ is an
abelian group, $(B,\cdot)$ is a group and
$$a\cdot(b+c)+a=a\cdot b+a\cdot c,$$
for all $a,b,c\in B$.

In \cite{BCJ}  a method is given to construct explicitly, given a
left brace $B$, all the solutions $(X,r)$ of the YBE such that
$\mathcal{G}(X,r)\cong B$ as left braces.  Therefore, the problem of
constructing all the finite solutions of the YBE is reduced to
describing all the finite left braces. In particular, as a
consequence of \cite[Theorem~3.1]{BCJ}, every finite left brace is
isomorphic to the left brace $\mathcal{G}(X,r)$, for some finite
solution $(X,r)$ of the YBE. Therefore the multiplicative group of
every finite left brace is solvable.  The finite solvable groups
isomorphic to the multiplicative group of a left brace are called
IYB groups \cite{CJR}. Not all the finite solvable groups are IYB
groups. In fact, there exist finite $p$-groups ($p$ a prime) that
are not of this type (see \cite{B2}). On the other hand, there are
some known classes of finite solvable groups that are  IYB groups,
(see \cite{BDG2,CJR}). It is an open problem whether every finite
nilpotent group of class three is an IYB group.

Every left brace $B$ has a left action $\lambda\colon
(B,\cdot)\longrightarrow \Aut(B,+)$, defined by
$\lambda(b)=\lambda_b$ and $\lambda_b(a)=b\cdot a-b$, for all
$a,b\in B$ (see \cite[Lemma~1]{CJO2}). It is called the lambda map
of the left brace $B$.  Note that, via the lambda map, the additive
structure of $B$ is determined by the multiplicative structure of
$B$ and vice-versa. It is known and easy to check that, in every
left brace $B$, $0=1$, i. e. the neutral elements of the groups
$(B,+)$ and $(B,\cdot)$ coincide. A left ideal of a left brace $B$
is a subgroup $I$ of $(B,+)$ such that $\lambda_b(a)\in I$, for all
$b\in B$ and all $a\in I$. Note that, in this case, for $x,y\in I$,
we have that
$$y^{-1}=-\lambda_{y^{-1}}(y)\in I\quad\mbox{and}\quad xy=\lambda_x(y)+x\in I,$$
thus $I$ also is a subgroup of $(B,\cdot)$. An ideal of a left brace
$B$ is a normal subgroup $N$ of $(B,\cdot)$ such that
$\lambda_b(a)\in N$, for all $b\in B$ and all $a\in N$. Note that
that, in this case, if $x,y\in N$, then
$$-y=\lambda_y(y^{-1})\in N\quad\mbox{and}\quad x+y=x\lambda_{x^{-1}}(y)\in N,$$
thus $N$ also is a subgroup of $(B,+)$, and hence it is a left ideal
of $B$. A non-zero left brace $B$ is simple if $\{ 0\}$ and $B$ are
the only ideals of $B$. We say that a left brace $B$ is trivial if
$a+b=a\cdot b$, for all $a,b\in B$.

In order to classify all finite left braces, it is natural to start
with classifying the finite simple left braces. To build a
foundation of the theory of finite simple braces it is crucial to
find various methods of constructions of such structures and to
understand their common features. In this paper, we continue a
systematic approach to this fundamental problem, initiated in
\cite{BCJO}. It is known that every simple left brace of prime power
order $p^n$ is a trivial brace of cardinality $p$, \cite[Corollary
on page 166]{Rump}. Until recently, these were the only known
examples of finite simple left braces. The first finite nontrivial
simple left braces have been constructed in \cite[Theorem~6.3 and
Section 7]{B3}; the additive groups of which are isomorphic  to
$\mathbb{Z}/(p_1)\times (\mathbb{Z}/(p_2))^{k(p_1-1)+1}$, where
$p_1,p_2$ are primes such that $p_2\mid p_1-1$  and $k$ is a
positive integer. In \cite{BCJO}, using the iterated matched product
construction, we built new families of finite simple left braces of
order $p_1^{n_1}\cdots p_s^{n_s}$, where $\{ p_1,\dots ,p_s\}$ is
any finite set of primes of cardinality greater that $1$ and
$n_1,\dots ,n_s$ are positive integers depending on the different
primes $p_1,\dots ,p_s$. It is remarkable that the derived length of
the multiplicative group of all the previously known simple left
braces is at most $3$.

There are some known restrictions on the possible cardinalities of
finite simple left braces (see \cite{Smokt}), but it is not at all
clear which cardinalities may occur.

In the study of finite simple groups the classification of all
minimal simple groups obtained by Thompson was of great significance
(see \cite[page~444]{gor}). Recall that a minimal simple group is a
finite non-commutative simple group all of whose proper subgroups
are solvable. Thus it seems natural to ask for the meaning of a
minimal non-trivial finite simple left brace. This should be a
non-trivial finite simple left brace such that all the factors of
the composition series of every proper subbrace are trivial. In
Section~\ref{solv}, we introduce the concept of solvable left brace
in a similar way as it is done in group theory. Using this concept,
a minimal non-trivial finite simple left brace is a non-trivial
simple left brace all of whose proper subbraces are solvable. We
expect that the classification of such finite simple left braces
will be important in the further study of finite simple left braces.

In Sections~\ref{asymmetric} and~\ref{wreathprod}, we recall the two
main tools of the approach presented in this paper: the asymmetric
product of left braces and the wreath product of left braces,
respectively. We also prove some technical results useful for the
construction of new simple left braces.

In Section~\ref{new}, using the asymmetric product, we construct new
classes of simple left braces. Furthermore, we construct first
examples of finite simple left braces with a multiplicative group
that is solvable of arbitrary derived length. This is accomplished
via the asymmetric product of special left braces, constructed using
the wreath product, and trivial braces.

In Section~\ref{known}, we interpret all previously known
constructions of finite left braces as asymmetric products.

\section{Solvable braces}\label{solv}

As motivated in the introduction, in this section we define solvable
left braces, a class of left braces including all the left nilpotent
left braces and all the right nilpotent left braces, introduced in
\cite{CGIS}. These concepts have some similarity with the
corresponding concepts in group theory.

Let $B$ be a left brace. Consider the operation $*$ of $B$ defined
by $a*b:=a\cdot b-a-b=(\lambda_a-\id)(b)$, for all $a,b\in B$. Rump
in \cite{Rump} introduced two series of subbraces of $B$. One
consists of left ideals $B^{n}$, defined inductively as follows: $B^1=B$ and
$$B^{n+1}=B*B^{n}=\{\sum_{i=1}^ma_i*b_i\mid a_i\in B,\; b_i\in B^{n},\; \mbox{ for every }i\},$$
for all positive integers $n$. Another one consisting of ideals
$B^{(n)}$, defined as follows. $B^{(1)}=B$ and
$$B^{(n+1)}=B^{(n)}*B=\{\sum_{i=1}^ma_i*b_i\mid a_i\in B^{(n)},\; b_i\in B,\; \mbox{ for every }i\},$$
for all positive integers $n$. Note that $B^{(2)}=B^2$, but in
general $B^{(3)}$ and $B^3$ are different because $*$ is not
associative in general.

We say that a left brace $B$ is left nilpotent if there exists a
positive integer $n$ such that $B^{n}=0$. A left brace $B$ is right
nilpotent if there exists a positive integer such that $B^{(n)}=0$.

\begin{remark}\label{nil}
{\rm Smoktunowicz in \cite[Theorem~1]{Smokt2} proved that a finite
left brace $B$ is left nilpotent if and only if its multiplicative
group $(B,\cdot)$ is nilpotent. In particular, every left brace of
order a power of a prime is left nilpotent.

On the other hand, a left brace $B$ is right nilpotent if and only
if the solution of the YBE associated with $B$ is a multipermutation
solution (see \cite[Proposition~6]{CGIS}). Recall, if $B$ is a left
brace, its associated solution of the YBE is $(B,r)$, where
$$r(a,b)=(\lambda_a(b),\lambda^{-1}_{\lambda_a(b)}(a)),$$
for all $a,b\in B$. In particular, every non-zero right nilpotent
left brace $B$ has non-zero socle, that is $\soc(B)=\{ a\in B\mid
a\cdot b=a+b\mbox{ for all } b\in B \}\neq \{ 0\}$.

In \cite{B} it is proved that there exist two left braces $B_1$ and
$B_2$ of order $8$ such that $\soc(B_i)=\{ 0\}$. Therefore $B_1$ and
$B_2$ are left nilpotent but not right nilpotent.

Consider the trivial left braces
$K=\mathbb{Z}/(2)\times\mathbb{Z}/(2)$ and $\mathbb{Z}/(3)$. Let
$\alpha\colon \mathbb{Z}/(3)\longrightarrow \Aut(K,+)$ be the action
defined by $\alpha(x)=\alpha_{x}$ and
$$\alpha_x(y,z)=(y,z)\left(\begin{array}{cc}
0&1\\
1&1\end{array}\right)^x,$$ for all $x\in \mathbb{Z}/(3)$ and all
$y,z\in \mathbb{Z}/(2)$. Then the semidirect product $B_3=K\rtimes
\mathbb{Z}/(3)$ is a left brace which is not left nilpotent. Now we
have
\begin{eqnarray*}
\lefteqn{((y_1,z_1),x_1)*((y_2,z_2),x_2)}\\
&=&((y_1,z_1),x_1)\cdot((y_2,z_2),x_2)-((y_1,z_1),x_1)-((y_2,z_2),x_2)\\
&=&((y_1,z_1)+\alpha_{x_1}(y_2,z_2),x_1+x_2)-((y_1,z_1),x_1)-((y_2,z_2),x_2)\\
&=&(\alpha_{x_1}(y_2,z_2)-(y_2,z_2),0),
\end{eqnarray*}
for all $y_i,z_i\in \mathbb{Z}/(2)$ and all $x_i\in \mathbb{Z}/(3)$.
Hence $B_3^{(2)}=K\times \{0\}$ and, since
$((y_1,z_1),0)*((y_2,z_2),x_2)=((0,0),0)$, we get that
$B_3^{(3)}=\{0\}$. Therefore $B_3$ is right nilpotent.}\end{remark}

Recall that a trivial brace is a left brace $B$ such that any
$a,b\in B$ satisfy $a\cdot b=a+b$. We also say that $B$ has trivial
structure. Thus $B$ is a trivial brace if and only if
$B^2=\{0\}$.

Let $B$ be a left brace. It was proven in \cite[page
161]{Rump} that $B^2$ is always an ideal of $B$. Observe that
$B/B^2$ has trivial brace structure because $a\cdot
b-a-b=a*b\in B^2$. Moreover, if $I$ is an ideal of $B$, $B/I$ has
trivial structure if and only if $B^2\subseteq I$.

In particular, note that $[(B,\cdot),(B,\cdot)]\leq B^2$ because $(B/B^2,\cdot)$ is abelian.

\begin{definition}
A \emph{solvable brace} is any left brace $B$
which has a series $0=B_0\subseteq B_1\subseteq \dots\subseteq B_m=B$ such that $B_i$ is an ideal of $B_{i+1}$, and such that
$B_{i+1}/B_i$ is a trivial brace for any $i\in\{0,1,\dots,m-1\}$.
\end{definition}

Using the same arguments as in group theory, one can prove the
following results.
\begin{proposition}{\rm (Second Isomorphism Theorem).}
Let $B$ be a left brace, let $H$ be a sub-brace of $B$, and let $N$ be an
ideal of $B$. Then, $HN$ is a sub-brace of $B$. Moreover,
$$
HN/N\cong H/(H\cap N).
$$
\end{proposition}

\begin{proposition}
\begin{enumerate}[(a)]
\item Define $d_1(B)=B^2$, and $d_{i+1}(B)=d_i(B)^2$ for every positive integer $i$. Then, $B$ is a solvable brace if and only if $d_k(B)=0$ for some $k$.
\item Let $B$ be a left brace, and let $I$ be an ideal of $B$. If $I$ and $B/I$ are solvable braces, then $B$ is also solvable.
\item Any sub-brace, and any quotient of a solvable brace is solvable.
\end{enumerate}
\end{proposition}

The third property means that there are neither simple sub-braces
nor simple quotients in a solvable brace, except for the trivial
simple braces $\Z/(p)$, where $p$ is a prime.

\begin{remark} {\rm Note that for every left brace $B$, $d_n(B)\subseteq
B^{n+1}\cap B^{(n+1)}$. Hence every left nilpotent left brace is
solvable, and every right nilpotent left brace is solvable. We have
seen in Remark~\ref{nil} that there exists a finite left nilpotent
left brace $B_1$ which is not right nilpotent, and there exists a
finite right nilpotent left brace $B_3$ which is not left nilpotent.
Hence the left brace $B_1\times B_3$ is neither left nilpotent nor
right nilpotent. But clearly $B_1\times B_2$ is
solvable.}\end{remark}

On the other extreme end, one can introduce the following class of
braces.

\begin{definition}
A left brace $B$ is called \emph{perfect} if it satisfies $B^2=B$.
\end{definition}

For instance, every non-trivial simple left brace is perfect. The
following example shows that the converse is not true.  So, this
class should play an important role in the approach towards a
description of all finite left simple braces, and should deserve a
further study.

\begin{example}
Take $B$ to be the simple left brace of order $24$ defined in
\cite{B3}. Since the multiplicative group of $B$ is isomorphic to
$S_4$, we can define a morphism $\alpha: (B,\cdot)\rightarrow
\aut(\Z/(3))$, given by the sign morphism of $S_4$. Then, the
semidirect product of braces $\Z/(3)\rtimes B$ with respect to
$\alpha$, where $\Z/(3)$ is considered as a trivial brace, is a
perfect brace which is not simple.
\end{example}

\section{Asymmetric product}\label{asymmetric}
In this section we recall a construction of left braces introduced
by Catino, Colazzo and  Stefanelli \cite{CCS}, the asymmetric
product of two left braces. Furthermore we give a result that
can be used to construct new simple left braces using asymmetric
products.

Let $S$ and $T$ be two (additive) abelian groups. Recall that a (normalized)
symmetric $2$-cocyle on $T$ with values in $S$ is a map $b\colon
T\times T\longrightarrow S$ such that
\begin{enumerate}[(i)]
\item\label{cond0} $b(0,0)=0$;
\item\label{cond1} $b(t_1,t_2)=b(t_2,t_1)$;
\item\label{cond2} $b(t_1+t_2,t_3)+b(t_1,t_2)=b(t_1,t_2+t_3)+b(t_2,t_3)$, for all $t_1,t_2,t_3\in T$.
\end{enumerate}
As a consequence, we get that $b(t,0)=b(0,t)=0$, for all $t\in T$.

\begin{theorem}\label{ccs}  (\cite[Theorem~3]{CCS}).
Let $T$ and $S$ be two left braces. Let $b \colon T\times T
\longrightarrow S$ be a symmetric $2$-cocycle on $(T, +)$ with
values in $(S, +)$, and let $\alpha \colon (S,\cdot)\longrightarrow
\Aut(T,+,\cdot)$ be a homomorphism of groups such that
\begin{eqnarray}\label{cond3} && s\cdot b(t_2, t_3) + b(t_1\cdot \alpha_s(t_2 + t_3),t_1)=
b(t_1\cdot\alpha_s(t_2), t_1\cdot\alpha_s(t_3))+ s,\end{eqnarray}
 where $\alpha_s=\alpha(s)$, for all $s\in S$ and $t_1, t_2, t_3\in T$. Then the addition and
multiplication over $T\times S$ given by
$$
(t_1,s_1)+(t_2,s_2)=(t_1+t_2,~s_1+s_2+b(t_1,t_2)),
$$
$$
(t_1,s_1)\cdot (t_2,s_2)=(t_1\cdot \alpha_{s_1}(t_2),~s_1\cdot s_2),
$$
define a structure of left brace on $T\times S$. We call this left
brace the asymmetric product of $T$ by $S$ (via $b$ and $\alpha$)
and denote it by  $T\rtimes_\circ S$.
\end{theorem}

Note that the lambda map of $T\rtimes_\circ S$ is defined by
\begin{eqnarray} \label{deflambda}
\lambda_{(t_1,s_1)}(t_2,s_2)&=&
\left(
\lambda_{t_1}\alpha_{s_1}(t_2),~
\lambda_{s_1}(s_2)-b(\lambda_{t_1}\alpha_{s_1}(t_2),t_1)
\right),
\end{eqnarray}
and its socle is
$$
\soc(T\rtimes_\circ S)=\{(t,s)\mid \lambda_s=\id_S,\;\lambda_t\circ
\alpha_s=\id_T,\; b(t,t')=0\mbox{ for all }t'\in T\}.
$$
Moreover, the subset $T\times\{0\}$ is a normal subgroup of $(T\rtimes_\circ S,\cdot)$, and $\{0\}\times S$ is a left
ideal of $T\rtimes_\circ S$.

A particular case of this theorem is when we assume that $b$ is a symmetric bilinear form. In this case, conditions
(\ref{cond0})-(\ref{cond2}) are automatic, and
condition (\ref{cond3}) becomes
$$
\lambda_s(b(t_2,~t_3))=b(\lambda_{t_1}\alpha_{s}(t_2),~\lambda_{t_1}\alpha_{s}(t_3)),
$$
which is equivalent to the two conditions:
\begin{align}
\lambda_s(b(t_2,~t_3))&=b(\alpha_{s}(t_2),~\alpha_{s}(t_3)),\label{condAlpha} \\
b(t_2,~t_3)&=b(\lambda_{t_1}(t_2),~\lambda_{t_1}(t_3)),\label{condLambda}
\end{align}
for all $s\in S$ and $t_1,t_2,t_3\in T$.

 To study simplicity of left braces we will often make use of the
following result without any further explicit reference.

\begin{lemma}(\cite[Lemma~2.5]{BCJO}).
Let $B$ be a left brace. Let $I$ be an ideal of $B$. Then
$$(\lambda_a-\id)(b)\in I,$$
for all $a\in I$ and all $b\in B$.
\end{lemma}

\begin{theorem}
Let $S$ be a simple non-trivial left brace. Let $T$ be a trivial
brace and let $T\rtimes_{\circ} S$ be an asymmetric product of $T$
by $S$ via $b\colon T\times T\longrightarrow S$ and $\alpha\colon
(S,\cdot)\longrightarrow \Aut(T,+)$. Suppose that for every $t\in
T\setminus\{ 0\}$ there exists $t'\in T$ such that $b(t,t')\neq 0$.
Then $T\rtimes_{\circ} S$ is simple if and only if $\sum_{s\in
S}\im(\alpha_s-\id)=T$.
\end{theorem}

\begin{proof}
Let $I$ be a non-zero ideal of $T\rtimes_{\circ} S$. Let $(t,s)\in
I$ be a nonzero element. Suppose that $s=0$. Then $t\neq 0$. There
exists $t'\in T$ such that $b(t,t')\neq 0$, and we have
\begin{eqnarray*}
(\lambda_{(t,0)}-\id)(t',0)&=&(t,0)(t',0)-(t,0)-(t',0)\\
&=&(t+t',0)-(t+t',b(t,t'))\\
&=&(0,-b(t,t'))\in I.
\end{eqnarray*}
Hence we may assume that $s\neq 0$. Since $S$ is a simple
non-trivial left brace, $\soc(S)=\{ 0\}$. So there exists $s'\in S$
such that $\lambda_s(s')\neq s'$. Note that
\begin{eqnarray*}
(\lambda_{(t,s)}-\id)(0,s')&=&(t,s)(0,s')-(t,s)-(0,s')\\
&=&(t,ss')-(t,s+s')\\
&=&(0,\lambda_s(s')-s')\in I\setminus\{ (0,0)\}.
\end{eqnarray*}
Since $S$ is a simple left brace, we easily get that $\{ (0,y)\mid
y\in S\}\subseteq I$. Let $x\in T$ and $y,z\in S$. We have
\begin{eqnarray*}
(\lambda_{(0,y)}-\id)(x,z)&=&(0,y)(x,z)-(0,y)-(x,z)\\
&=&(\alpha_y(x),yz)-(x,y+z)\\
&=&((\alpha_y-\id)(x),yz-y-z-b(x,(\alpha_y-\id)(x)))\in I
\end{eqnarray*}
and thus $((\alpha_y-\id)(x),0)\in I$, because $\{ (0,y)\mid y\in
S\}\subseteq I$.

Let $T'=\sum_{y\in S}\im(\alpha_y-\id)$. We have that $T'\times
S\subseteq I$. It is clear that $T'\times S$ is a subgroup of the
additive group of $T\rtimes_{\circ} S$. Let $(x,y)\in
T\rtimes_{\circ} S$ and $(x',y')\in T'\times S$. Note that
\begin{eqnarray*}
\lambda_{(x,y)}(x',y')&=&(x,y)(x',y')-(x,y)\\
&=&(x+\alpha_y(x'),yy')-(x,y)\\
&=&(\alpha_y(x'),yy'-y-b(x,\alpha_y(x')))
\end{eqnarray*}
and $\alpha_y\circ(\alpha_z-\id)=(\alpha_{yzy^{-1}}-\id)\circ
\alpha_y$. Therefore $\lambda_{(x,y)}(x',y')\in T'\times S$.
Furthermore
\begin{eqnarray*}
(x,y)(x',y')(x,y)^{-1}&=&(x+\alpha_y(x'),yy')(-\alpha_y^{-1}(x),y^{-1})\\
&=&(x+\alpha_y(x')-\alpha_{yy'y^{-1}}(x),yy'y^{-1})\in T'\times S.
\end{eqnarray*}
Hence $T'\times S$ is an ideal of $T\rtimes_{\circ} S$. Thus
$T\rtimes_{\circ} S$ is simple if and only if $T'=T$, and the result
follows.
\end{proof}

\section{Wreath products of left braces}\label{wreathprod}
In this section we recall the definition of the wreath product of
left braces (see \cite[Corollary~1]{CJO2}) and we study how to
construct automorphisms of wreath products of left braces. The aim
is to build new classes of left braces via the iterated wreath
product construction.

Let $G_1$ and $G_2$ be left braces. Recall that the wreath product
$G_2\wr G_1$ of the left braces $G_2$ and $G_1$ is a left brace
which is the semidirect product of left braces $H_2\rtimes G_1$,
where $H_2=\{ f\colon G_1\longrightarrow G_2\mid |\{g\in G_1\mid
f(g)\neq 1\}|<\infty\}$ is a left brace with the operations
$(f_1\cdot f_2)(g)=f_1(g)\cdot f_2(g)$ and
$(f_1+f_2)(g)=f_1(g)+f_2(g)$, for all $f_1,f_2\in H_2$ and $g\in
G_1$, and the action of $(G_1,\cdot)$ on $H_2$ is given by the
homomorphism $\sigma\colon (G_1,\cdot)\longrightarrow \Aut(H_2,
+,\cdot)$ defined by $\sigma(g)(f)(x)=f(g^{-1}x)$, for all $g,x\in
G_1$ and $f\in H_2$. We will denote $\sigma(g)(f)$ simply by
$g(f)$.

As in group theory  (see \cite[Hilfssatz~I.15.7(b)]{Huppert}), one
can extend every automorphism of the left brace $G_1$ to an
automorphism of $G_2\wr G_1$ in a natural way.

\begin{proposition}\label{wreath}
Let $G_1$ and $G_2$ be two left braces. Let $\alpha_1\in
\Aut(G_1,+,\cdot)$. Let $H_2=\{ f\colon G_1\longrightarrow G_2\mid
|\{g\in G_1\mid f(g)\neq 1\}|<\infty\}$. Let
$\widehat{\alpha}_1\colon H_2\longrightarrow H_2$ be the map defined
by $\widehat{\alpha}_1(f)(g)=f(\alpha_1^{-1}(g))$, for all $f\in
H_2$ and all $g\in G_1$. Then the map $\alpha_2\colon G_2\wr
G_1\longrightarrow G_2\wr G_1$ defined by
$\alpha_2(f,g)=(\widehat{\alpha}_1(f),\alpha_1(g))$, for all $f\in
H_2$ and all $g\in G_1$, is an automorphism of left braces.
\end{proposition}

\begin{proof}
The proof is similar to the proof of
\cite[Hilfssatz~I.15.7(b)]{Huppert} and we leave it to the reader.
\end{proof}

In the situation described above, we say that $\alpha_2$ is
the automorphism of $G_2\wr G_1$ induced by the automorphism
$\alpha_1$ of $G_1$. Thus, given left braces $G_1,G_2,\dots ,G_n$,
every automorphism $\alpha_1$ of the left brace $G_1$, induces an
automorphism $\alpha_n$ of the left brace
$$G_n\wr(G_{n-1}\wr\cdots(G_2\wr G_1)\cdots ),$$
where $\alpha_i$ is the automorphism of
$G_i\wr(G_{i-1}\wr\cdots(G_2\wr G_1)\cdots )$ induced by the
automorphism $\alpha_{i-1}$ of $G_{i-1}\wr(G_{i-2}\wr\cdots(G_2\wr
G_1)\cdots )$ (as in Proposition~\ref{wreath}), for all $n\geq i\geq
2$.

\begin{proposition}\label{composition}
Let $G_1$ and $G_2$ be two left braces. Let $\alpha_1,\beta_1\in
\Aut(G_1,+,\cdot)$. Let $\alpha_2,\beta_2$ be the automorphisms of
the left brace $G_2\wr G_1$ induced by $\alpha_1,\beta_1$,
respectively. Then the automorphism of $G_2\wr G_1$ induced
by $\alpha_1\circ\beta_1$ is $\alpha_2\circ\beta_2$.
\end{proposition}

\begin{proof}
Let $\gamma_2$ be the automorphism of $G_2\wr G_1$ induced by
$\gamma_1=\alpha_1\circ\beta_1$. Let $H_2=\{ f\colon
G_1\longrightarrow G_2\mid |\{g\in G_1\mid f(g)\neq 1\}|<\infty\}$.
Let $\widehat{\gamma}_1\colon H_2\longrightarrow H_2$ be the map
defined by $\widehat{\gamma}_1(f)(g)=f(\gamma_1^{-1}(g))$, for all
$f\in H_2$ and all $g\in G_1$. We have that
\begin{eqnarray*}
&&\gamma_2(f,g)=(\widehat{\gamma}_1(f),\gamma_1(g)),
\end{eqnarray*}
for all $f\in H_2$ and all $g\in G_1$. Note that
\begin{eqnarray*}
\widehat{\gamma}_1(f)(h)&=&f(\beta_1^{-1}(\alpha_1^{-1}(h)))\\
&=&\widehat{\beta}_1(f)(\alpha_1^{-1}(h))\\
&=&\widehat{\alpha}_1(\widehat{\beta}_1(f))(h),
\end{eqnarray*}
for all $f\in H_2$ and all $h\in G_1$. Hence
$\widehat{\gamma}_1=\widehat{\alpha}_1\circ\widehat{\beta}_1$. Now
it is clear that $\gamma_2=\alpha_2\circ\beta_2$ and the result
follows.
\end{proof}

\begin{lemma}\label{semidirect}
Let $G$ be a finite group and let $K$ be a field of characteristic
$p$. Suppose that $p$ is not a divisor of the order of $G$. Consider
the semidirect product $K[G]\rtimes G$ with respect to the action of
$G$ on the additive group of $K[G]$ by left multiplication.  Let
$D_0=G$ and, for $m>0$, let $D_m=[D_{m-1},D_{m-1}]$. Let
$E_0=K[G]\rtimes G$ and, for $m>0$, let $E_m=[E_{m-1},E_{m-1}]$.
 Then
\begin{equation}\label{claim}
E_m=\omega(K[D_{m-1}])K[G]\rtimes D_{m},
\end{equation}
for all $m>0$. (Here $\omega(K[G])$ denotes the augmentation ideal
of the group algebra $K[G]$.) In particular, the derived length of $E_0$ is $d+1$, where $d$ is the derived length of $G$.
\end{lemma}

\begin{proof}
We shall prove the result by induction on $m$.
For $m=1$, we have
that
$$((1-h)\alpha,1)=(\alpha, h^{-1})^{-1}(0,h)^{-1}(\alpha,h^{-1})(0,h),$$
for all $h\in G$ and $\alpha\in K[G]$, and
$$(0,h^{-1}g^{-1}hg)=(0,h)^{-1}(0,g)^{-1}(0,h)(0,g).$$ Hence
$$\omega(K[G])\rtimes D_{1}\subseteq E_1.$$
On the other hand
\begin{eqnarray*}
\lefteqn{(\alpha,h)^{-1}(\beta,g)^{-1}(\alpha,h)(\beta,g)}\\
&=&(-h^{-1}\alpha,h^{-1})(-g^{-1}\beta,g^{-1})(\alpha+h\beta,hg)\\
&=&(-h^{-1}\alpha-h^{-1}g^{-1}\beta,h^{-1}g^{-1})(\alpha+h\beta,hg)\\
&=&(-h^{-1}\alpha-h^{-1}g^{-1}\beta+h^{-1}g^{-1}(\alpha+h\beta),h^{-1}g^{-1}hg)\\
&=&((h^{-1}g^{-1}-h^{-1})\alpha+(h^{-1}g^{-1}h-h^{-1}g^{-1})\beta,h^{-1}g^{-1}hg),
\end{eqnarray*}
for all $\alpha,\beta\in K[G]$ and all $h,g\in G$. Hence
$E_1\subseteq \omega(K[G])\rtimes D_{1}$. Therefore $E_1=
\omega(K[G])\rtimes D_{1}$.

Suppose that $m>1$ and that (\ref{claim}) holds for $m-1$. Let
$\alpha\in \omega(K[D_{m-2}])K[G]$. We have
$$((1-h)\alpha,1)=(\alpha,
h^{-1})^{-1}(0,h)^{-1}(\alpha,h^{-1})(0,h),$$ for all $h\in
D_{m-1}$. Thus
\begin{eqnarray*}\lefteqn{\omega(K[D_{m-1}])\omega(K[D_{m-2}])K[G]\times\{ 1\} }\\
&&\subseteq [\omega(K[D_{m-2}])K[G]\rtimes
D_{m-1},\omega(K[D_{m-2}])K[G]\rtimes D_{m-1}].
\end{eqnarray*}
Since $p$ is not a divisor of $|G|$, the group algebras $K[D_{j}]$
are semisimple by Maschke's Theorem, for all non-negative integers
$j$. Hence $$\omega(K[D_{m-1}])K[D_{m-2}]=
\omega(K[D_{m-1}])\omega(K[D_{m-2}]),$$ and consequently
$$\omega(K[D_{m-1}])K[G] = \omega(K[D_{m-1}])\omega(K[D_{m-2}])K[G],$$
if $m>1$.
Hence
\begin{eqnarray*}
\lefteqn{\omega(K[D_{m-1}])K[G]\rtimes D_{m}}\\
&\subseteq&
[\omega(K[D_{m-2}])K[G]\rtimes D_{m-1},\omega(K[D_{m-2}])K[G]\rtimes
D_{m-1}]\\
&=&[E_{m-1},E_{m-1}]\quad(\mbox{by the induction hypothesis})\\
&=& E_m.
\end{eqnarray*}
Since
\begin{eqnarray*}\lefteqn{(\alpha,h)^{-1}(\beta,g)^{-1}(\alpha,h)(\beta,g)}\\
&&=((h^{-1}g^{-1}-h^{-1})\alpha+(h^{-1}g^{-1}h-h^{-1}g^{-1})\beta,h^{-1}g^{-1}hg),
\end{eqnarray*}
for all $\alpha,\beta\in \omega(K[D_{m-2}])K[G]$ and all $h,g\in
D_{m-1}$, we also have that
$$E_m\subseteq\omega(K[D_{m-1}])K[G]\rtimes
D_{m},$$ and therefore $E_m=\omega(K[D_{m-1}])K[G]\rtimes D_{m}$.
Thus the result follows.
\end{proof}

\section{New simple left braces}\label{new}

In \cite{B3}, Bachiller gave the first examples of non-trivial
finite simple left braces.  They are  of order $p_1^{n_1}p_2$ for
any pair of primes $p_1,p_2$  such that $p_1$ divides $p_2-1$, and
$n_1$ is a positive integer that depends on $p_2$.
 In \cite{BCJO}, for
every $s>1$ and every set $\{ p_1,\dots , p_s\}$ of $s$ distinct
primes, the authors constructed examples of finite simple left
braces of order $p_1^{n_1}\cdots p_s^{n_s}$ for some positive
integers $n_1,\dots ,n_s$ depending on the $p_i$'s.

In this section we generalize the construction of simple left braces
of \cite{B3}. Moreover, using the iterated wreath product construction, we
give examples of finite simple left
braces with multiplicative group of arbitrary derived length. Note that the
multiplicative groups of all
the previously known finite simple left braces \cite{B3,BCJO} have derived
length at most $3$.

\subsection{Finite simple left braces with multiplicative group of arbitrary derived length}

 Let $R$ be a ring. We will consider the trivial brace associated
to $R$, which is $(R,+,+)$. Let $G$ be a left brace. Note that the
trivial left brace
$$H=\{f\colon G\longrightarrow R\mid |\{
x\in G\mid f(x)\neq 0\}|<\infty\}$$ can be identified with the
trivial left brace associated to the group ring $R[G]$ by
$$f=\sum_{x\in G}f(x)x,$$
for all $f\in H$. We know that the wreath product of left braces
$R\wr G$ is the semidirect product $H\rtimes G$ with respect to
the action defined by $g(f)(x)=f(g^{-1}x)$ for all $x,g\in G$ and
$f\in H$. Note that with the above identification this action is
$$g(f)=\sum_{x\in G}f(g^{-1}x)x=\sum_{x\in G}f(x)gx=g\sum_{x\in G}f(x)x,$$
so it coincides with the left multiplication by $g\in G$ in the
group ring $R[G]$. Thus $R\wr G$ can be identified with the
semidirect product $R[G]\rtimes G$ with respect to the action
given by left multiplication.

Let $n$ be an integer greater than $1$. Let $p_1,\dots ,p_n$ be
distinct prime numbers. Consider the trivial braces associated to
the fields $\mathbb{F}_{p_i}$ of order $p_i$, for $i=1,\dots ,n$.
Sometimes we will need to use the ring product of
$\mathbb{F}_{p_i}$. In order to avoid confusion, we always denote
this product by a dot.

 We shall define recursively two series of left braces $G_j$ and
$\overline{G}_j$, for $j=1,\ldots,n$, in such a way that
$\overline{G}_j$ is defined as a quotient of $G_j$, and that
$G_{j+1}$ is defined as a wreath product. First of all, let
$G_1=\overline{G}_1=\F_{p_1}$. By recurrence we define the
$G_j$'s, for $2\leq j\leq n$, as $G_j=\F_{p_j}\wr
\overline{G}_{j-1}$. For this to make sense, we must define next
$\overline{G}_{j}$ as a quotient of $G_{j}$. Let $H_j=\{ f\colon
\overline{G}_{j-1}\longrightarrow \F_{p_j}\}$ be the set of all
the functions from $\overline{G}_{j-1}$ to $\F_{p_j}$. Since
$\F_{p_j}$ is a trivial brace, so is $H_j$ (for the pointwise
addition of functions). Let $I_j=\{ f\colon
\overline{G}_{j-1}\longrightarrow \F_{p_j}\mid f\text{ is a
constant map} \}$. Clearly $I_j$ is an ideal of $H_j$. By the
definition of wreath product, we know that $G_j=H_j\rtimes
\overline{G}_{j-1}$. Let
$$N_j=\{ (f,0)\in G_j\mid f\in I_j\}.$$
Note that $N_j$ is a central subgroup of the multiplicative group
of the left brace $G_j$. Indeed, for $(f_1,0)\in N_j$ and
$(f_2,g)\in G_j$, we have
$$(f_2,g)(f_1,0)=(f_2+g(f_1),g)=(f_2+f_1,g)=(f_1,0)(f_2,g),$$
where the second equality holds because $f_1$ is a constant map.
Furthermore, since
$$\lambda_{(f_1,0)}(f_2,g)=(f_1,0)(f_2,g)-(f_1,0)=(f_1+f_2,g)- (f_1,0) =(f_2,g),$$
we have that $N_j\subseteq\soc(G_j)$. Thus $N_j$ is an ideal of
the left brace $G_j$. We define $\overline{G}_j=G_j/N_j$. Let
$\overline{H}_j=H_j/I_j$. It is clear that
$\overline{G}_j=\overline{H}_j\rtimes \overline{G}_{j-1}$. Hence
$$\overline{G}_{n}=\overline{H}_{n}\rtimes(\overline{H}_{n-1}\rtimes(\dots\rtimes(\overline{H}_{2}\rtimes \F_{p_1})\dots )),$$
and, therefore, every element in $\overline{G}_{n}$ is of the form
$$(\overline{h}_n,\overline{h}_{n-1},\dots ,\overline{h}_{2},c),$$
where $\overline{h}_{j}\in\overline{H}_{j}$ denotes the class of
$h_j\in H_j$ modulo $I_j$, and $c\in \F_{p_1}$.

Note that $|\overline{G}_1|=|\F_{p_1}|=p_1$ and, for $2\leq j\leq
n$,
$$|\overline{G}_j|=\frac{|G_j|}{|N_j|}=\frac{|\overline{G}_{j-1}|p_j^{|\overline{G}_{j-1}|}}{p_j}.$$
Thus $|\overline{G}_j|=p_1p_2^{m_2}\cdots p_j^{m_j}$, for some
positive integers $m_2,\dots ,m_j$.

\paragraph{} Now we introduce some actions and some bilinear forms
in order to apply the asymmetric product construction to the
braces $\overline{G}_j$. Suppose that $p_j$ is a divisor of
$p_{i}-1$ for all $2\leq j\leq n$ and all $1\leq i<j$. The
existence of such primes is a consequence of Dirichlet's Theorem
on arithmetic progressions. Let $A$ be the trivial brace
associated to the ring $\F_{p_2}\times \dots \times \F_{p_n}$.
Then $a=(1,\dots ,1)$ is a generator of the trivial brace $A$.
Besides the brace structure, $A$ is a ring, and we will use this
structure sometimes. Again we will denote this ring product as a
dot.

Since $p_j$ is a divisor of $p_1-1$, for all $2\leq j\leq n$,
there exists $\gamma\in (\F_{p_1})^*$  of order $p_2\cdots p_n$.
Let $\alpha_1(i a)$ be the automorphism of the left brace
$\F_{p_1}$ defined by $\alpha_1(i a)(z)=\gamma^{i}\cdot z$, for
all $z\in \F_{p_1}$, and $i\in\mathbb{Z}$. Since the additive
order of $a$ is $p_2\cdots p_n$, $\alpha_1(ia)$ is well-defined.
By Proposition~\ref{wreath}, $\alpha_1(ia)$ induces an
automorphism $\alpha_2(ia)$ of the left brace $G_2$. Recall that
$$\alpha_2(ia)(h_2,c)=(\widehat{\alpha}_1(ia)(h_2),\alpha_1(ia)(c)),$$
and
$$\widehat{\alpha}_1(ia)(h_2)(d)=h_2(\alpha_1(ia)^{-1}(d)).$$

It is clear that $\alpha_2(ia)(N_2)=N_2$. Thus $\alpha_2(ia)$
induces an automorphism $\alpha'_2(ia)$ of $\overline{G}_2$. By
Proposition~\ref{composition}, $\alpha_2(ia)=\alpha_2(a)^{i}$ and
we also have $\alpha'_2(ia)=\alpha'_2(a)^{i}$. Hence, the map
$\alpha'_2\colon A\longrightarrow \Aut(\overline{G}_2,+,\cdot)$ is
a morphism of groups. By induction one can see that $\alpha_1(ia)$
induces an automorphism $\alpha'_j(ia)$ of the left brace
$\overline{ G}_j$, for all $2\leq j\leq n$, and the map
$\alpha'_j\colon A\longrightarrow \Aut(\overline{G}_j,+,\cdot)$ is
a morphism of groups.  Note that
$$\alpha'_j(ia)(\overline{h}_j,\overline{g}_{j-1})=
\left(\overline{\widehat{\alpha}_{j-1}(ia)(h_j)}~,~\alpha'_{j-1}(ia)(\overline{g}_{j-1})\right),$$
for all $\overline{h}_j\in \overline{H}_j$ and all
$\overline{g}_{j-1}\in \overline{G}_{j-1}$,  where
$$\widehat{\alpha}_{j-1}(ia)(h_j)(\overline{g})=h_j\left(\alpha'_{j-1}(ia)^{-1}(\overline{g})\right),$$
for all $\overline{g}\in \overline{G}_{j-1}$.

Next we are going to define some bilinear forms. For $f\in H_j$
put $\varepsilon_{j}(f) =\sum_{g\in \overline{G}_{j-1}} f(g)$. Let
$b_j\colon H_j\times H_j\longrightarrow \F_{p_j}$ be the symmetric
bilinear form defined by
$$ b_j(f_1,f_2)=\varepsilon_{j}(f_1)\cdot \varepsilon_{j}(f_2) - \varepsilon_{j}(f_1\bullet f_2),$$
for all $f_1,f_2\in H_j$, where $f_1\bullet f_2\in H_j$ is defined
by $(f_1\bullet f_2)(x)=f_1(x)\cdot f_2(x)$, for all
$x\in\overline{G}_{j-1}$. Recall that $\varepsilon_{j}(f_1)\cdot
\varepsilon_{j}(f_2)$ denotes the ring product of $\F_{p_j}$.

Let $\mathds{1}\colon \overline{G}_{j-1}\longrightarrow \F_{p_j}$
be the map defined by $\mathds{1}(g)=1$, for all $g\in
\overline{G}_{j-1}$. Note that
$$b_j(\mathds{1},f_2)=(|\overline{G}_{j-1}|-1)\varepsilon_{j}(f_2)=0,$$
because $p_j$ is a divisor of $|\overline{G}_{j-1}|-1$.  Thus
$b_j(I_j,H_j)=0$ and then $b_j$ induces a symmetric bilinear form
$b'_j\colon \overline{H}_j\times \overline{ H}_j\longrightarrow
\F_{p_j}$. It is easy to see that $b'_j$ is non-degenerate. We
define  $b\colon \overline{G}_n\times
\overline{G}_n\longrightarrow A$  by
$$b((\overline{h}_n,\overline{h}_{n-1},\dots,\overline{h}_2,z_1),(\overline{h'}_n,\overline{h'}_{n-1},\dots,\overline{h'}_2,z_1'))
=(b'_2(\overline{h}_2,\overline{h'}_2),\dots,b'_n(\overline{h}_n,\overline{h'}_n)),$$
for all $\overline{h}_j,\overline{h'}_j\in \overline{H}_j$, $2\leq
j\leq n$, and all $z_1,z'_1\in \F_{p_1}$. It is easy to see that
$b$ is bilinear.

In order to apply Theorem~\ref{ccs}, we need to check that the
bilinear form $b$ satisfies conditions (\ref{condAlpha}) and
(\ref{condLambda}). As a first step, we will check the analogous
property for $b_j$.

\begin{lemma}\label{pre-bilinear}
Let $h_1,h_2\in H_j$ and $g\in \overline{G}_{j-1}$. Then
$$b_j(h_1,h_2)=b_j(\widehat{\alpha}_{j-1}(a)(h_1)~,~\widehat{\alpha}_{j-1}(a)(h_2)),$$
(recall that $a=(1,\dots ,1)$ is  a generator of the trivial brace
$A$) and
$$b_j(h_1,h_2)=b_j(g(h_1)~,~g(h_2)),$$
 where $g(h)\colon
\overline{G}_{j-1}\longrightarrow \F_{p_j}$ is the map defined by
$(g(h))(g')=h(g^{-1}g')$, for all $g'\in\overline{G}_{j-1}$ and
$h\in H_j$.

\end{lemma}
\begin{proof}
The right side of the first equality is by definition
\begin{align*}
b_j(\widehat{\alpha}_{j-1}(a)(h_1),\widehat{\alpha}_{j-1}(a)(h_2))
=\;&\varepsilon_{j}(\widehat{\alpha}_{j-1}(a)(h_1))\cdot \varepsilon_{j}(\widehat{\alpha}_{j-1}(a)(h_2))\\
&- \varepsilon_{j}(\widehat{\alpha}_{j-1}(a)(h_1)\bullet
\widehat{\alpha}_{j-1}(a)(h_2)).
\end{align*}
The action of $\widehat{\alpha}_{j-1}(a)$ over a map $f$ is a
permutation of its images. Thus $\varepsilon_j(f)$, which is the
sum of all the images of $f$, is invariant under this action.
Therefore,
\begin{align*}
b_j(\widehat{\alpha}_{j-1}(a)(h_1),\widehat{\alpha}_{j-1}(a)(h_2))
=\;&\varepsilon_{j}(\widehat{\alpha}_{j-1}(a)(h_1))\cdot \varepsilon_{j}(\widehat{\alpha}_{j-1}(a)(h_2)) \\
&- \varepsilon_{j}(\widehat{\alpha}_{j-1}(a)(h_1)\bullet \widehat{\alpha}_{j-1}(a)(h_2))\\
=\;&\varepsilon_{j}(h_1)\cdot \varepsilon_{j}(h_2) - \varepsilon_{j}(h_1\bullet h_2)\\
=\;& b_j(h_1,h_2).
\end{align*}

The proof of the second equality is analogous, because the action
of $g\in \overline{G}_{j-1}$ over a map is also a permutation of
its images.
\end{proof}

From the previous lemma, we are going to deduce the desired result
for $b$.

\begin{lemma}\label{bilinear}
Let $g_1,g_2,x\in \overline{G}_n$. Then
$$b(g_1,g_2)=b(\alpha'_n(a)(g_1),\alpha'_n(a)(g_2)),$$
and
$$b(g_1,g_2)=b(\lambda_{x}(g_1),\lambda_{x}(g_2)).$$
(Recall that $a=(1,\dots,1)$ is a generator of the trivial brace
$A$.)
\end{lemma}

\begin{proof}
The first equality is a direct consequence of
Lemma~\ref{pre-bilinear}, and the de\-fi\-ni\-tion of $b$ in terms
of all the $b'_j$, plus the definition of $\alpha'_n$ in terms of
all the $\widehat{\alpha}_j$.

For the second equality, recall that $x$ is of the form
$$
x=(\overline{h}_{n},\overline{h}_{n-1},\dots,\overline{h}_{2},z)
$$
for some $\overline{h}_{j}\in\overline{ H}_j$ and $z\in \F_{p_1}$.
We use the following notation: let $x_{1}=z$ and
$x_{j}=(\overline{h}_{j},
\overline{h}_{j-1},\dots,\overline{h}_{2},z)$, for $j=2,\dots
,n-1$. Then, if $g_k$, $k=1,2$, is of the form
$$g_k=(\overline{h}_{k,n},\overline{h}_{k,n-1},\dots,\overline{h}_{k,2},z_{k}),$$
the corresponding lambda map is equal to
\begin{eqnarray*}\lambda_{x}(g_k)&=&xg_k-x\\
&=&(\overline{h}_{n}+
\overline{x_{n-1}(h_{k,n})}-\overline{h}_{n},\dots
,\overline{h}_{2}+\overline{x_{1}(h_{k,2})}-\overline{h}_{2},~z+ z_k-z)\\
&=&(\overline{x_{n-1}(h_{k,n})},
\overline{x_{n-2}(h_{k,n-1})},\dots , \overline{x_{1}(h_{k,2})},~
z_k),
\end{eqnarray*}
where $x_{j-1}(h_{k,j})\colon \overline{G}_{j-1}\longrightarrow
\F_{p_j}$ is the map defined by
$$(x_{j-1}(h_{k,j}))(g)=h_{k,j}(x_{j-1}^{-1}g),$$
for all
$g\in\overline{G}_{j-1}$, for $j=2,\dots,n$. By
Lemma~\ref{pre-bilinear},
$$b_j(x_{j-1}(h_{1,j}),x_{j-1}(h_{2,j}))=b_j( h_{1,j},h_{2,j}).$$ Hence
\begin{eqnarray*}
b(\lambda_{x}(g_1),\lambda_{x}(g_2))&=&
(b'_2(\overline{x_{1}(h_{1,2})},\overline{x_{1}(h_{2,2})}),\dots ,b'_n(\overline{x_{n-1}(h_{1,n})},\overline{x_{n-1}(h_{2,n})}))\\
&=&(b_2(x_{1}(h_{1,2}),x_{1}(h_{2,2})),\dots ,b_n(x_{n-1}(h_{1,n}),x_{n-1}(h_{2,n})))\\
&=&(b_2(h_{1,2},h_{2,2}),\dots ,b_n(h_{1,n},h_{2,n}))\\
&=&(b'_2(\overline{h}_{1,2},\overline{h}_{2,2}),\dots ,b'_n(\overline{h}_{1,n},\overline{h}_{2,n}))\\
&=&b(g_1,g_2).
\end{eqnarray*}
Thus, the result is proved.
\end{proof}
Since $A$ is a trivial brace, by Lemma~\ref{bilinear} and
Theorem~\ref{ccs}, we can construct the asymmetric product
$\overline{ G}_n\rtimes_{\circ}A$ (via $b$ and $\alpha'_n$).

Recall from (\ref{deflambda}) that, for $(g,c),(g',c')\in
\overline{G}_n\rtimes_{\circ}A$,
$$\lambda_{(g,c)}(g',c')=
\left( \lambda_{g}\alpha'_n(c)(g')~,~
c'-b(\lambda_{g}\alpha'_n(c)(g'),g) \right).$$

Consider the semidirect product of left braces
$\F_{p_j}[\overline{G}_{j-1}]\rtimes \overline{G}_{j-1}$ with
respect to the action of $\overline{G}_{j-1}$ on the additive
group of the group ring $\F_{p_j}[\overline{G}_{j-1}]$ (considered
as a trivial brace) by left multiplication. As mentioned above, we
may identify $\F_{p_j}\wr \overline{G}_{j-1}$ with
$\F_{p_j}[\overline{G}_{j-1}]\rtimes \overline{G}_{j-1}$.

\begin{theorem}
The left brace $\overline{G}_n\rtimes_{\circ}A$ is simple and the
derived length of its multiplicative group is $n+1$.
\end{theorem}

\begin{proof}
 In this proof we shall identify $H_j$ with
$\F_{p_j}[\overline{G}_{j-1}]$ (via $f=\sum_{g\in
\overline{G}_{j-1}}f(g)g$, for all $f\in H_j$). Thus
$H_j=\F_{p_j}[\overline{G}_{j-1}]$,
$I_j=\{\sum_{g\in\overline{G}_{j-1}}zg\mid z\in \F_{p_j}\}$. As
usual, we  denote by $\overline{h}$ the natural image of $h\in
H_j$ in $\overline{H}_j=H_j/I_j$.

Note also that, for $j=2,\dots ,n$, $H_j$, $\overline{H}_j$ and
$A$ are trivial braces. In order to avoid confusion, we only use
the sum operation for these trivial braces. We will only use the
multiplication of the trivial brace $G_1$ for elements in the
group ring $H_2$. Thus the product in $H_j$ will mean the product
in the corresponding group ring.

We shall prove the simplicity of the left brace
$\overline{G}_n\rtimes_{\circ}A$  by induction on $n$. For $n=2$,
let $I$ be a nonzero ideal of $\overline{G}_2\rtimes_{\circ}A$.
Let $(\overline{h}_2,z_1,c)\in I$ be a nonzero element, where
$\overline{h}_2\in \overline{H}_2$, $z_1\in G_1$ and $c\in A$ . We
consider three mutually exclusive cases.

{\em Case 1.} Suppose that $\overline{h}_2\neq 0$. Then
$p_1(\overline{h}_2,z_1,c)=(p_1\overline{h}_2,0,c')\in I$, for
some $c'\in A$ and $p_1\overline{h}_2\neq 0$, because
$\overline{H}_2$ has order a power of $p_2$. Thus we may assume
that $(\overline{h}_2,0,c)\in I$ and $\overline{h}_2\neq 0$. Since
$b'_2$ is non-degenerate, there exists $\overline{h}\in
\overline{H}_2$ such that $b'_2(\overline{h}_2,\overline{h})\neq
0$. Hence
\begin{eqnarray*}
\lambda_{(\overline{h},0,0)}(\overline{h}_2,0,c)-(\overline{h}_2,0,c)&=&(\overline{h}_2,0,c-b'_2(\overline{h}_2,\overline{h}))-(\overline{h}_2,0,c)\\
&=&(0,0,-b'_2(\overline{h}_2,\overline{h}))\in I.
\end{eqnarray*}
Since $A$ has order $p_2$ and $b'(\overline{h}_2,\overline{h})\in
A\setminus\{ 0\}$, we have that $\{0\}\times \{0\}\times
A\subseteq I$. In particular $(0,0,a)\in I$,  where $a=1$ is the
fixed generator of $A$. Thus
$$\lambda_{(0,0,a)}(0,z,0)-(0,z,0)=(0,(\gamma-1)z,0)\in I,$$
for all $z\in G_1$. Since $\gamma-1$ is a nonzero element of
$\F_{p_1}$, we have $\{0\}\times G_1\times \{ 0\}\subseteq I$.
Hence $\{0\}\times G_1\times A\subseteq I$. Now we have
$$(\lambda_{(0,z,0)}-\id)(\overline{h},0,0)=(\overline{(z-1)h},0,-b'_2(\overline{(z-1)h},\overline{h}))\in I,$$
for all $\overline{h}\in \overline{H}_2$ and $z\in G_1$. Hence
$$(\sum_{z\in G_1}\overline{(z-1)h},0,0)\in I,$$
for all $h\in H_2$. Therefore $\pi_2(\omega(\F_{p_2}[G_1]))\times
G_1\times A\subseteq I$, where $\pi_2\colon H_2\longrightarrow
\overline{H}_2$ is the natural map. Note that
$$I_2+\omega(\F_{p_2}[G_1])=\F_{p_2}[G_1]=H_2.$$
Therefore $\overline{H}_2\times G_1\times A\subseteq I$. This
implies that $\overline{G}_2\rtimes_{\circ}A=I$ in this case.

{\em Case 2.} Suppose that $\overline{h}_2=0$ and $z_1\neq 0$. Now
$p_2(0,z_1,c)=(0,p_2z_1,0)\in I$ and $p_2z_1\neq 0$ because $G_1$
has order $p_1$. Now
$$\lambda_{(0,z_1,0)}(\overline{h},0,0)-(\overline{h},0,0)=(\overline{z_1h-h},0,-b'_2(\overline{z_1h-h},\overline{h}))\in I,$$
for all $\overline{h}\in \overline{H}_2$. In particular, for $h=1$
in the group ring $H_2$, we have
$(\overline{z_1-1},0,-b'_2(z_1-1,1))\in I$. Note  $z_1-1\in
H_2\setminus I_2$. Therefore $\overline{z_1-1}\neq 0$ in
$\overline{H}_2$. So, by Case 1,
$\overline{G}_2\rtimes_{\circ}A=I$, as desired.

{\em Case 3.} Suppose that $(h_2,z_1)=0$ and $c\neq 0$. Then
$(0,0,a)\in I$ and, as in the proof of   Case 1, we again get that
$\overline{G}_2\rtimes_{\circ}A=I$.

Therefore for $n=2$ the left brace
$\overline{G}_2\rtimes_{\circ}A$ is simple.

Suppose now that $n>2$ and assume that the result holds for $n-1$.
Let $A_{n-1}=\F_{p_2}\times\dots\times \F_{p_{n-1}}$. Note that
$\{0\}\times \overline{G}_{n-1}\times A_{n-1}\times \{0\}$ is a
subbrace of $\overline{G}_n\rtimes_{\circ}A$ isomorphic to
$\overline{G}_{n-1}\rtimes_{\circ}A_{n-1}$ which is simple by the
induction hypothesis. Let $I$ be a nonzero ideal of
$\overline{G}_n\rtimes_{\circ}A$ and let $(\overline{h}_n,\dots
,\overline{h}_2,z_1,c)$ be a nonzero element of $I$, where
$\overline{h}_j\in \overline{H}_j$, $z_1\in G_1$ and $c\in A$. We
again consider three mutually exclusive cases.

{\em Case i.} Suppose that $\overline{h}_n\neq 0$. Then
$$p_1\cdots
p_{n-1}(\overline{h}_n,\dots ,\overline{h}_2,z_1,c)=(p_1\cdots
p_{n-1}\overline{h}_n,0,\dots ,0,0,c')\in I$$ and $p_1\cdots
p_{n-1}\overline{h}_n\neq 0$, because $\overline{H}_n$ has order a
power of $p_n$. Thus, without loss of generality,  we may assume
that $(\overline{h}_n,0,\dots ,0,0,c)\in I$ and
$\overline{h}_n\neq 0$. Since $b'_n$ is non-degenerate, there
exists $\overline{h}\in \overline{H}_n$ such that
$b'_n(\overline{h}_n,\overline{h})\neq 0$. Hence
\begin{eqnarray*}\lefteqn{\lambda_{(\overline{h},0,\dots ,0,0)}(\overline{h}_n,0,\dots
,0,c)-(\overline{h}_n,0,\dots,0,c)}\\
&=&(\overline{h}_n,0,\dots ,0,c-(0,\dots,0,b'_n(\overline{h}_n,\overline{h})))-(\overline{h}_n,0,\dots ,0,c)\\
&=&(0,0,\dots
,0,-(0,\dots,0,b'_n(\overline{h}_n,\overline{h})))\in I.
\end{eqnarray*}
In particular $(0,\dots ,0,(0,\dots ,0,1))\in I$. Thus
\begin{eqnarray*}\lefteqn{(\lambda_{(0,\dots,0,(0,\dots,0,1))}-\id)(0,\dots
,0,z,0)}\\
&&=(0,\dots ,0,\alpha_1(0,\dots,0,1)(z)-z,0)\in I\cap (\{0\}\times
\overline{G}_{n-1}\times A_{n-1}\times\{ 0\}),
\end{eqnarray*}
for all $z\in G_1$. Since $\alpha_1(0,\dots,0,1)\neq \id$  and the
subbrace $\{0\}\times \overline{G}_{n-1}\times A_{n-1}\times\{
0\}$ is simple, we have that $\{0\}\times \overline{G}_{n-1}\times
A_{n-1}\times\{ 0\}\subseteq I$. Moreover, $\{0\}\times
\overline{G}_{n-1}\times A\subseteq I$ because $(0,\dots
,0,(0,\dots,0,1))\in I$. Now we have
$$(\lambda_{(0,g,0)}-\id)(\overline{h},0,0)=(\overline{gh-h},0,-(0,\dots ,0,b'_n(\overline{(g-1)h},\overline{h})))\in I,$$
for all $\overline{h}\in \overline{H}_n$ and $g\in
\overline{G}_{n-1}$, and then
$\pi_n(\omega(\F_{p_n}[\overline{G}_{n-1}]))\times \overline{
G}_{n-1}\times A\subseteq I$, where $\pi_n\colon
H_n\longrightarrow \overline{H}_n$ is the natural map. Note that
\begin{equation}\label{augment}I_n+\omega(\F_{p_n}[\overline{G}_{n-1}])=\F_{p_n}[\overline{G}_{n-1}]=H_n. \end{equation}
Therefore $\overline{G}_n\rtimes_{\circ}A=I$ in this case.

{\em Case ii.} Suppose that $\overline{h}_n=0$ and
$(\overline{h}_{n-1},\dots ,\overline{h}_2,z_1)\neq 0$. Since the
subbrace $\{0\}\times \overline{G}_{n-1}\times A_{n-1}\times\{ 0\}$
is simple, we have that $\{0\}\times\overline{G}_{n-1}\times
A_{n-1}\times\{ 0\}\subseteq I$. Hence, as in the proof of the last
part of Case i, $\overline{G}_n\rtimes_{\circ}A=I$.

{\em Case iii.} Suppose that $(\overline{h}_n,\dots
,\overline{h}_2,z_1)=0$ and $c\neq 0$. In this case
$$(\lambda_{(0,\dots,0,c)}-\id)(0,\dots ,0,z,0)=(0,\dots ,0,\alpha_1(c)(z)-z,0)\in I,$$
for all $z\in \F_{p_1}$. Since $\alpha_1(c)\neq \id$, by Case ii,
$\overline{G}_n\rtimes_{\circ}A=I$, again as desired.

Therefore $\overline{G}_n\rtimes_{\circ}A$ is simple.

In the remainder of the proof we  show that the derived length of
the multiplicative group of the left brace
$\overline{G}_n\rtimes_{\circ}A$ is $n+1$. Recall that this  group
is the semidirect product $\overline{G}_n\rtimes A$ (via the
action $\alpha'_n$) of the multiplicative groups of the left
braces $\overline{G}_n$ and $A$ (recall that $A$ is the trivial
brace associated to the ring $\F_{p_2}\times\dots\times \F_{p_n}$,
and therefore the multiplicative group of the trivial brace $A$
coincides with its additive group). First we shall prove by
induction on $n$ that
$$[\overline{G}_n\rtimes A,\overline{G}_n\rtimes A]=\overline{G}_n\times\{ 0\}.$$
Suppose that $n=1$. Note that if $z_1\in G_1$, then
$$[(1,a)^{-1},(z_1,0)^{-1}]=(1,a)(z_1,0)(1,a)^{-1}(z_1,0)^{-1}=((\gamma-1)\cdot z_1,0).$$
Hence $[G_1\rtimes A,G_1\rtimes A]=G_1\times \{0\}$. Suppose that
$n\geq 2$ and that
$$[\overline{G}_{n-1}\rtimes A_{n-1},\overline{G}_{n-1}\rtimes A_{n-1}]=\overline{G}_{n-1}\times \{0\}.$$
Let $g_{n-1}\in \overline{G}_{n-1}$ and $\overline{h_n}\in
\overline{ H}_n$. Since the left brace $H_n$ is trivial, we will
use the additive notation on $H_n$. Then
$$[(0,g_{n-1},0),(\overline{h}_n,1,0)]=(\overline{(1-g^{-1}_{n-1})h_n},1,0).$$
Thus
$$\pi_n(\omega(\F_{p_n}[\overline{G}_{n-1}]))\times\{1\}\times
\{0\}\subseteq [\overline{G}_{n}\rtimes A,\overline{G}_{n}\rtimes
A] .$$ By (\ref{augment}), which is also true for $n=2$, we obtain
 that
$$\overline{H}_n\times\{1\}\times \{0\}\subseteq [\overline{G}_{n}\rtimes A,\overline{G}_{n}\rtimes A].$$ Since
$$[\overline{G}_{n-1}\rtimes A_{n-1},\overline{G}_{n-1}\rtimes A_{n-1}]=\overline{G}_{n-1}\times \{0\},$$
we get that
$$\overline{H}_n\times \overline{G}_{n-1}\times \{0\}\subseteq [\overline{G}_{n}\rtimes A,\overline{G}_{n}\rtimes A]$$
and thus
$$[\overline{G}_{n}\rtimes
A,\overline{G}_{n}\rtimes A]=\overline{G}_{n}\times \{0\},$$ as
claimed. Hence, it is enough to show that the derived length of
the multiplicative group $\overline{G}_n$ is $n$. It is clear that
the derived length of $\overline{G}_2$ is $2$. Suppose that $n>2$
and that $\overline{G}_{n-1}$ has derived length $n-1$. Recall
that $G_n=\F_{p_n}[\overline{G}_{n-1}]\rtimes \overline{G}_{n-1}$.
Hence, by Lemma~\ref{semidirect}, the derived length of $G_n$ is
$n$. Note that $N_n=I_n\times \{1\}$ is a central subgroup of
$G_n=H_n\rtimes\overline{G}_{n-1}$, and by (\ref{augment})
$$G_n=N_n(\omega(\F_{p_n}[\overline{G}_{n-1}])\rtimes \overline{G}_{n-1}).$$
Furthermore, it is easy to see that
$$N_n\cap(\omega(\F_{p_n}[\overline{G}_{n-1}])\rtimes \overline{G}_{n-1})=\{1\}.$$
Since $$\overline{G}_n=G_n/N_n\cong
\omega(\F_{p_n}[\overline{G}_{n-1}])\rtimes \overline{G}_{n-1},$$
and
$$[G_n,G_n]=[\omega(\F_{p_n}[\overline{G}_{n-1}])\rtimes \overline{G}_{n-1},\omega(\F_{p_n}[\overline{G}_{n-1}])\rtimes \overline{G}_{n-1}]\cong
[\overline{G}_n,\overline{G}_n],$$ the derived length of
$\overline{G}_n$ also is $n$ and the result follows.
\end{proof}

\subsection{A generalized construction of simple braces}

 In this
subsection we present a broad class of simple left braces,
extending the construction of \cite{B3}. Let
$A=\mathbb{Z}/(l_1)\times\dots\times \mathbb{Z}/(l_s)$ be a finite
abelian group and let $p$ be a prime such that $p\mid q-1$ for
every prime divisor $q$ of the order of $A$. For every $i=1,\dots
,s$, let
\begin{enumerate}
\item
$b_i$ be a non-degenerate symmetric bilinear form over
$(\mathbb{Z}/(p))^{n_i}$,
\item  $f_i$ and $c_i$ be elements of order $p$ and $l_i$ respectively in
the orthogonal group of $b_i$,
\item $\gamma_i\in \mathbb{Z}/(l_i)$ be an invertible element of order $p$.
\end{enumerate}
Assume that
\begin{itemize}
\item[(i)]
$f_ic_i=c_i^{\gamma_i}f_i$,
\item[(ii)] $\gamma_i-1$ is invertible in $\mathbb{Z}/(l_i)$.
\end{itemize}
Put $n=n_1+\dots+n_s$. Let $T=(\mathbb{Z}/(p))^{n}\rtimes A$ be the
semidirect product of the trivial braces $(\mathbb{Z}/(p))^{n}$ and
$A$ via the action
$$\beta\colon A\longrightarrow \Aut((\mathbb{Z}/(p))^{n})$$
defined by $\beta_{(a_1,\dots ,a_s)}(u_1,\dots
,u_s)=(c_1^{a_1}(u_1),\dots,c_s^{a_s}(u_s))$, where $\beta(a_1,\dots
,a_s)=\beta_{(a_1,\dots ,a_s)}$, for all $a_i\in \mathbb{Z}/(l_i)$
and all $u_i\in (\mathbb{Z}/(p))^{n_i}$.

\begin{lemma}\label{alpha}
The map $\alpha\colon (\mathbb{Z}/(p),+)\longrightarrow
\Aut(T,+,\cdot)$ defined by
$$\alpha_{\mu}(u_1,\dots ,u_s,a_1,\dots
,a_s)=(f_1^{\mu}(u_1),\dots ,f_s^{\mu}(u_s),\gamma_1^{\mu}a_1,\dots
,\gamma_s^{\mu}a_s),$$ where $\alpha(\mu)=\alpha_{\mu}$, for all
$\mu\in\mathbb{Z}/(p)$, all $a_i\in \mathbb{Z}/(l_i)$ and all
$u_i\in (\mathbb{Z}/(p))^{n_i}$, is a homomorphism of groups.
\end{lemma}

\begin{proof}
It is clear that $\alpha_{\mu}\in \Aut(T,+)$. Because
$f_ic_i=c_i^{\gamma_i}f_i$ for all $i=1,\dots ,s$, it is
straightforward to verify that $\alpha_{\mu}$ also is a
multiplicative group homomorphism, and thus a brace homomorphism.
Hence, $\alpha$ is a homomorphism.
\end{proof}

Let $b\colon T\times T\longrightarrow \mathbb{Z}/(p),$ defined by
$$b((u_1,\dots ,u_s,a_1,\dots ,a_s),(v_1,\dots ,v_s,a'_1,\dots
,a'_s))=\sum_{i=1}^sb_i(u_i,v_i),$$
for  all  $u_i\in (\mathbb{Z}/(p))^{n_i}$ and all $a_i\in
\mathbb{Z}/(l_i)$.  Since
each $b_i$ is a symmetric bilinear form, it is clear that $b$ is a
symmetric $2$-cocyle on $T$ with values in $\mathbb{Z}/(p)$.

We shall see that $\alpha$ and $b$ satisfy condition  (\ref{cond3}). In
this case this is equivalent to verify conditions (\ref{condAlpha}) and
(\ref{condLambda}), and for this we need to check:
\begin{eqnarray}
b(t_2, t_3) = b(\alpha_{\mu}(t_2),
\alpha_{\mu}(t_3)),\end{eqnarray}
\begin{eqnarray}
b(t_2,t_3)= b(\lambda_{t_1}(t_2),\lambda_{t_1}(t_3)),\end{eqnarray}
for all $t_1,t_2,t_3\in T$ and all $\mu\in \mathbb{Z}/(p)$. That
these equalities hold follows from the definitions of $\alpha$,
$\lambda$ and $b$, and from the fact that $c_i$ and $f_i$ are in
the orthogonal group of $b_i$. Hence, by Theorem~\ref{ccs}, we get
the asymmetric product $T\rtimes_{\circ}\mathbb{Z}/(p)$ of $T$ by
$\mathbb{Z}/(p)$ via $b$ and $\alpha$.

Let $B=T\rtimes_{\circ} \mathbb{Z}/(p)$. Recall that the addition in
$B$ is given by
\begin{eqnarray*}\lefteqn{ (u_1,\dots, u_s,a_1,\dots
a_s,\mu)+(v_1,\dots, v_s,a'_1,\dots
a'_s,\mu')}\\
&=&\left(u_1+v_1,\dots, u_s+v_s,~~a_1+a'_1,\dots
a_s+a'_s,~~\mu+\mu'+\sum_{i=1}^sb_i(u_i,v_i)\right),\end{eqnarray*} for all
$u_i,v_i\in (\mathbb{Z}/(p))^{n_i}$, all $a_i,a'_i\in
\mathbb{Z}/(l_i)$ and all $\mu,\mu'\in \mathbb{Z}/(p)$. The lambda
map of $B$ is given by
\begin{eqnarray*}
\lefteqn{\lambda_{(u_1,\dots ,u_s,a_1,\dots ,a_s,\mu)}(v_1,\dots ,v_s,a'_1,\dots ,a'_s,\mu')}\\
&=&\left(c_1^{a_1}f_1^{\mu} (v_1),\dots
,c_s^{a_s}f_s^{\mu}(v_s),\gamma_1^{\mu}a'_1,\dots
,\gamma_s^{\mu}a'_s,\;
\mu'-\sum_{i=1}^sb_i(c_i^{a_i}f_i^{\mu}(v_i),u_i)\right),
\end{eqnarray*}
where $a_i,a'_i\in \mathbb{Z}/(l_i)$ and $u_i,v_i\in
(\mathbb{Z}/(p))^{n_i}$, $\mu,\mu'\in \mathbb{Z}/(p)$.

The following technical remark is of use for the subsequent proofs.

\begin{remark}\label{cifi}{\rm
For $i=1,\dots ,s$, let $G_i$ be the subgroup of
$\Aut((\mathbb{Z}/(p))^{n_i})$ generated by $c_i$ and $f_i$. Then
there exists a natural homomorphism of $\mathbb{Z}/(p)$-algebras
$\varphi_i\colon \mathbb{Z}/(p)[G_i]\longrightarrow
\End((\mathbb{Z}/(p))^{n_i})$ such that $\varphi_i(c_i)=c_i$ and
$\varphi_i(f_i)=f_i$. Since
\begin{eqnarray*}\omega(\mathbb{Z}/(p)[G_i])&=&(c_i-\id)\mathbb{Z}/(p)[G_i]+(f_i-\id)\mathbb{Z}/(p)[G_i]\\
&=&\mathbb{Z}/(p)[G_i](c_i-\id)+\mathbb{Z}/(p)[G_i](f_i-\id),
\end{eqnarray*}
it is clear that for every $u_i\in (\mathbb{Z}/(p))^{n_i}$ and every
$g\in \varphi(\omega(\mathbb{Z}/(p)[G_i]))$, $g(u_i)\in
\im(c_i-\id)+\im(f_i-\id)$. Furthermore, if $v_i\in
\im(c_i-\id)+\im(f_i-\id)$, then $f_i(v_i),c_i(v_i)\in
\im(c_i-\id)+\im(f_i-\id)$.}
\end{remark}

\begin{lemma}\label{ideal}
Let
\begin{eqnarray*}
 I'&=&\{ (w_1,\dots,w_s,a'_1,\dots,a'_s,\mu')\mid w_i \in
\im(c_i-\id)+\im(f_i-\id),\; a'_i\in \mathbb{Z}/(l_i), \\
 &&  \mbox{ for } i=1, \ldots, s,\; \mu'\in
\mathbb{Z}/(p) \}.
\end{eqnarray*}
Then $I'$ is an ideal of $B$.
\end{lemma}

\begin{proof}
Clearly $I'$ is a subgroup of the additive group of $B$.

Let $(u_1,\dots ,u_s,a_1,\dots ,a_s,\mu)\in B$ and $(w_1,\dots
,w_s,a'_1,\dots ,a'_s,\mu')\in I'$. We have
\begin{eqnarray*}\lefteqn{\lambda_{(u_1,\dots ,u_s,a_1,\dots ,a_s,\mu)}(w_1,\dots ,w_s,a'_1,\dots ,a'_s,\mu')}\\
&=&(c_1^{a_1}f_1^{\mu} (w_1),\dots
,c_s^{a_s}f_s^{\mu}(w_s),\gamma_1^{\mu}a'_1,\dots
,\gamma_s^{\mu}a'_s,\;
\mu'-\sum_{i=1}^sb_i(c_i^{a_i}f_i^{\mu}(v_i),u_i)).
\end{eqnarray*}
Thus, by Remark~\ref{cifi},  $\lambda_{(u_1,\dots
,u_s,a_1,\dots ,a_s,\mu)}(w_1,\dots ,w_s,a'_1,\dots ,a'_s,\mu')\in
I'$. Finally
\begin{eqnarray*}
\lefteqn{(u_1,\dots ,u_s,a_1,\dots ,a_s,\mu)^{-1}(w_1,\dots
,w_s,a'_1,\dots ,a'_s,\mu')(u_1,\dots ,u_s,a_1,\dots ,a_s,\mu)}\\
&&=(-f_1^{-\mu}c_1^{-a_1}(u_1),\dots
,-f_s^{-\mu}c_s^{-a_s}(u_s),-\gamma_1^{-\mu}a_1,\dots
,-\gamma_s^{-\mu}a_s,-\mu)\\
&&\quad(w_1+c_1^{a'_1}f_1^{\mu'}(u_1),\dots
,w_s+c_s^{a'_s}f_s^{\mu'}(u_s),a'_1+\gamma_1^{\mu'}a_1,\dots
,a'_s+\gamma_s^{\mu'}a_s,\mu'+\mu)\\
&&=(-f_1^{-\mu}c_1^{-a_1}(u_1)+c_1^{-\gamma_1^{-\mu}a_1}f_1^{-\mu}(w_1+c_1^{a'_1}f_1^{\mu'}(u_1)),\dots
,-f_s^{-\mu}c_s^{-a_s}(u_s)\\
&&\quad+c_s^{-\gamma_s^{-\mu}a_s}f_s^{-\mu}(w_s+c_s^{a'_s}f_s^{\mu'}(u_s)),-\gamma_1^{-\mu}a_1\\
&&\quad+\gamma_1^{-\mu}(a'_1+\gamma_1^{\mu'}a_1),\dots
,-\gamma_1^{-\mu}a_1+\gamma_1^{-\mu}(a'_1+\gamma_1^{\mu'}a_1),\mu')\in
I',
\end{eqnarray*}
by Remark~\ref{cifi}. Thus, $I'$ is a normal subgroup of the
multiplicative group of $B$, and the result follows.
\end{proof}

\begin{theorem}\label{main}  The left brace $B=T\rtimes_{\circ} \mathbb{Z}/(p)$ is simple if and only if
$\im(c_i-\id)+\im(f_i-\id)=(\Z/(p))^{n_i}$, for every $i=1,\dots
,s$. In particular, a sufficient condition to be simple is that
$c_i-\id$ is invertible for every $i=1,\dots ,s$.
\end{theorem}

\begin{proof} It is enough to prove that any non-zero ideal of $B$ contains
the ideal $I'$ defined in the previous lemma.
So let $I$ be a nonzero ideal of $B$. Let $(u_1,\dots ,u_s,a_1,\dots
,a_s,\mu)$ be a nonzero element of $I$.
 We consider two cases.

 {\it Case 1:}
suppose that some $a_i$ is nonzero. Then
\begin{eqnarray*}p(u_1,\dots,u_s,a_1,\dots,a_s,\mu)&=&(pu_1,\dots,pu_s,pa_1,\dots
,pa_s,\mu')\\
&=&(0,\dots ,0,pa_1,\dots ,pa_s,\mu')\in I,\end{eqnarray*} for some
$\mu'\in \mathbb{Z}/(p)$, and $pa_i\neq 0$ because the order of
$A$ is not a multiple of the prime $p$.  Now
$$p(0,\dots,0,pa_1,\dots,pa_s,\mu')=(0,\dots,0,p^2a_1,\dots
,p^2a_s,0)\in I,$$ and $p^2a_i\neq 0$. Thus we may assume that
$(0,\dots ,0,a_1,\dots,a_s,0)\in I$ and $a_i\neq 0$. Let $v_i\in
(\mathbb{Z}/(p))^{n_i}$ and for $j\neq i$, let $v_j=0\in
(\mathbb{Z}/(p))^{n_j}$. Then
\begin{eqnarray*}\lefteqn{(v_1,\dots,v_s,0,\dots ,0,0)^{-1}(0,\dots
,0,a_1,\dots,a_s,0)(v_1,\dots,v_s,0,\dots ,0,0)}\\
&=&(-v_1,\dots,-v_s,0,\dots,0,0)(c_1^{a_1}(v_1),\dots,c_s^{a_s}(v_s),a_1,\dots ,a_s,0)\\
&=&((c_1^{a_1}-\id)(v_1),\dots,(c_s^{a_s}-\id)(v_s),a_1,\dots,a_s,0)\in
I.
\end{eqnarray*}
Since $c_i$ has order $l_i$ and $a_i\neq 0$, there exists $v_i\in
(\mathbb{Z}/(p))^{n_i}$ such that $(c_i^{a_i}-\id)(v_i)\neq 0$. Now
\begin{eqnarray*}\lefteqn{\!\!\!\!\!\!\!\!\!\!\!\! ((c_1^{a_1}-\id)(v_1),\dots,(c_s^{a_s}-\id)(v_s),a_1,\dots,a_s,0)-(0,\dots,0,a_1,\dots,a_s,0)}\\
&&=((c_1^{a_1}-\id)(v_1),\dots,(c_s^{a_s}-\id)(v_s),0,\dots,0,0)\in
I.\end{eqnarray*} Thus we may assume that
$(0,\dots,0,u_i,0,\dots,0,0,\dots, 0,0)\in I$ for some nonzero
element $u_i\in (\mathbb{Z}/(p))^{n_i}$. Since $b_i$ is
non-degenerate, there exists $w_i\in (\mathbb{Z}/(p))^{n_i}$ such
that $b_i(w_i,u_i)\neq 0$. We have that
\begin{eqnarray*}
\lefteqn{(\lambda_{(0,\dots,0,u_i,0,\dots,0,0,\dots,
0,0)}-\id)(0,\dots,0,w_i,0,\dots,0,0,\dots, 0,0)}\\
&&=(0,\dots,0,0,\dots,0,-b_i(w_i,u_i))\in I.
\end{eqnarray*}
Hence $\{(0,\dots,0,0,\dots,0,\mu')\mid \mu'\in
\mathbb{Z}/(p)\}\subseteq I$. Note that for arbitrary $a_{1}', \ldots , a_{n}'$,
\begin{eqnarray*}
\lefteqn{(\lambda_{(0,\dots ,0,0,\dots,
0,1)}-\id)(0,\dots,0,a'_1,\dots,a'_n,0)}\\
&=&(0,\dots,0,(\gamma_1-1)a'_1,\dots,(\gamma_s-1)a'_s,0)\in I.
\end{eqnarray*}
Since, by assumption,  $\gamma_j-1$ is invertible in
$\mathbb{Z}/(l_j)$, we have that
$$\{(0,\dots ,0,a'_1,\dots ,a'_s,0)\mid a'_j\in \mathbb{Z}/(l_i)\}\subseteq I.$$
Hence
$$\{(0,\dots ,0,a'_1,\dots ,a'_s,\mu')\mid a'_j\in \mathbb{Z}/(l_i),\; \mu'\in \mathbb{Z}/(p)\}\subseteq I.$$
Now we have that
\begin{eqnarray*}\lefteqn{(\lambda_{(0,\dots
,0,a'_1,\dots ,a'_s,\mu')}-\id)(v_1,\dots,v_s,0,\dots,0,0)}\\
&=&((c_1^{a'_1}f_1^{\mu'}-\id) (v_1),\dots
,(c_s^{a'_s}f_s^{\mu'}-\id)(v_s),0,\dots,0,0)\in I,
\end{eqnarray*}
for all $v_j\in (\mathbb{Z}/(p))^{n_j}$, all $a'_j\in
\mathbb{Z}/(l_i)$ and all $\mu'\in \mathbb{Z}/(p)$.  Hence
\begin{eqnarray*}
 I'&=&\{(w_1,\dots,w_s,a'_1,\dots,a'_s,\mu')\mid w_j\in \im(c_i-\id)+\im(f_i-\id),\\
  && \;\; a'_j\in
\mathbb{Z}/(l_i),\; \mu'\in \mathbb{Z}/(p)\}\subseteq I,
\end{eqnarray*}
 as desired.

 {\it Case 2:} suppose that $a_j=0$ for all $j$. Thus $(u_1,\dots, u_s,0,\dots
,0,\mu)\in I$ is a nonzero element. If $\mu\neq 0$, then
$$(\lambda_{(u_1,\dots, u_s,0,\dots
,0,\mu)}-\id)(0,\dots ,0,1,\dots ,1,0)=(0,\dots
,0,\gamma_1^{\mu}-1,\dots ,\gamma_s^{\mu}-1,0)\in I.$$ Therefore, by Case 1,
 $I'\subseteq I$. If $\mu=0$, then some $u_i\neq 0$ and
$$\lambda_{(0,\dots ,0,w_i,0,\dots ,0,0\dots ,0,0)}(u_1,\dots, u_s,0,\dots
,0,0)=(0,\dots ,0,0,\dots ,0,-b_i(w_i,u_i))\in I.$$ Since $b_i$ is
non-degenerate, there exists $w_i\in (\mathbb{Z}/(p))^{n_i}$ such
that $b_i(w_i,u_i)\neq 0$. Therefore, by the first part of Case 2,
$I'\subseteq I$. Thus the result follows.
\end{proof}

\subsection{Concrete construction of simple left braces}

In this section we give concrete realisations of simple braces as
constructed in Theorem~\ref{main}.

Let $m$ be a positive integer and let $p,p_1,p_2,\dots, p_{m}$ be
different primes such that $p\mid p_j-1$, for all $j=1,\dots,m$. Let
$s$ be a positive integer and let
$A=\mathbb{Z}/(l_1)\times\dots\times \mathbb{Z}/(l_s)$, where each
$l_i$ is an integer such that $l_i>1$ and
$$l_i=p_1^{r_{i,1}}\cdots p_m^{r_{i,m}},$$
for some nonnegative integers $r_{i,j}$.

Consider the natural isomorphism of rings
$$\mathbb{Z}/(l_i)\cong \mathbb{Z}/(p_1^{r_{i,1}})\times\cdots\times \mathbb{Z}/(p_m^{r_{i,m}}).$$
Since $p\mid p_j-1$, for all $j=1,\dots,m$, if $r_{i,j}>0$, then
there exist invertible elements $\gamma_{i,j}\in
\mathbb{Z}/(p_j^{r_{i,j}})$ of order $p$. If $r_{i,j}=0$, let
$\gamma_{i,j}=0$. Let $\gamma_i\in \mathbb{Z}/(l_i)$ be the element
corresponding to $(\gamma_{i,1},\dots ,\gamma_{i,m})$ under the
natural isomorphism. Then, $\gamma_i$ is an invertible element of
order $p$ and $\gamma_i-1$ corresponds to the element
$(\gamma_{i,1}-1,\dots ,\gamma_{i,m}-1)$.  Note that if $r_{i,j}>0$
then $(1+p_{j})^{p_{j}^{r_{i,j}-1}}=1$ in
$\mathbb{Z}/(p_j^{r_{i,j}})$. Hence, in this case,
$\gamma_{i,j}-1\notin p_j\mathbb{Z}/(p_j^{r_{i,j}})$, and
consequently, $\gamma_{i,j}-1$ is invertible in
$\mathbb{Z}/(p_j^{r_{i,j}})$. Therefore, $\gamma_i-1$ is invertible
in $\mathbb{Z}/(l_i)$.

For $i=1,\dots ,s$, define the ring
$R_i=\mathbb{F}_p[x]/(x^{l_i-1}+\cdots+x+1)$, and denote
$\xi_i:=\overline{x}\in R_i$. Let $c_i\in \Aut(R_i,+)$ be defined by
$$c_i(r_i)=\xi_ir_i,$$
for all $r_i\in R_i$. Let $f_i\in \Aut(R_i,+)$ be defined by
$$f_i(r_i(\xi_i))=r_i(\xi_i^{\gamma_i}),$$
for all $r_i(\xi_i)\in R_i$. Note that
$$f_ic_i(r_i(\xi_i))=f_i(\xi_ir_i(\xi_i))=\xi_i^{\gamma_i}r_i(\xi_i^{\gamma_i})=c_i^{\gamma_i}f_i(r_i(\xi_i)),$$
for all $r_i(\xi_i)\in R_i$. Hence $f_ic_i=c_i^{\gamma_i}f_i$.

Note that
\begin{align*}
l_i-1&=p_1^{r_{i,1}}\cdots p_m^{r_{i,m}}-1\\
&=p_1^{r_{i,1}}\cdots p_m^{r_{i,m}}-p_1^{r_{i,1}}\cdots
p_{m-1}^{r_{i,m-1}}+p_1^{r_{i,1}}\cdots p_{m-1}^{r_{i,m-1}}-1\\
&=p_1^{r_{i,1}}\cdots
p_{m-1}^{r_{i,m-1}}(p_m^{r_{i,m}}-1)+p_1^{r_{i,1}}\cdots
p_{m-1}^{r_{i,m-1}}-1\\
&=p_1^{r_{i,1}}\cdots
p_{m-1}^{r_{i,m-1}}(p_m^{r_{i,m}}-1)+p_1^{r_{i,1}}\cdots
p_{m-2}^{r_{i,m-2}}(p_{m-1}^{r_{i,m-1}}-1)+\dots + p_1^{r_{i,1}}-1.
\end{align*}
Since $p\mid p_j-1$, for all $j=1,\dots ,m$, we have that $p\mid
l_i-1$, for all $i$. Since the characteristic polynomial of $c_i$ is
$x^{l_i-1}+\cdots+x+1\in \mathbb{F}_p[x]$ and $p\mid l_i-1$,
$c_i-\id$ is invertible.

Let $b_i\colon R_i\times R_i\longrightarrow \mathbb{F}_p$ be the
unique bilinear form such that
$b_i(\xi_i^j,\xi_i^k)=1-\delta_{j,k}$, for all $j,k\in
\{0,1,\dots,l_i-2\}$. Note that $(1,\xi_i,\dots ,\xi_i^{l_i-2})$ is a
basis of the $\mathbb{F}_p$-vector space $R_i$. The matrix of $b_i$
with respect to this basis is
$$\begin{pmatrix}
0& 1& \cdots &1\\
1& \ddots & \ddots & \vdots\\
\vdots& \ddots & \ddots & 1\\
1&\cdots & 1& 0
\end{pmatrix}\in M_{l_i-1}(\mathbb{F}_p).$$
Since $p\mid l_i-1$, the determinant of this matrix is
$(l_i-2)(-1)^{l_i-2}=(-1)^{l_i-1}\in \mathbb{F}_p$. Hence $b_i$ is
non-degenerate. Furthermore
$$b_i(f_i(\xi_i^{j}),f_i(\xi_i^k))=b_i(\xi_i^{j\gamma_i},\xi_i^{k\gamma_i})=1-\delta_{j,k}=b_i(\xi_i^j,\xi_i^k),$$
and
\begin{eqnarray*}
&&b_i(c_i(\xi_i^j),c_i(\xi_i^k))=b_i(\xi_i^{j+1},\xi_i^{k+1})=1-\delta_{j,k}\quad\mbox{(if
 $0\leq j,k<l_i-2$),}\\
&&b_i(c_i(\xi_i^{l_i-2}),c_i(\xi_i^k))=b_i(\xi_i^{l_i-1},\xi_i^{k+1})=-\sum_{t=0}^{l_i-2}(1-\delta_{t,k+1})=-(l_i-2)=1\\
&&\quad\mbox{(if
 $0\leq k<l_i-2$) and}\\
&&b_i(c_i(\xi_i^{l_i-2}),c_i(\xi_i^{l_i-2}))=b_i(\xi_i^{l_i-1},\xi_i^{l_i-1})=\sum_{t,t'=0}^{l_i-2}(1-\delta_{t,t'})\\
&&\hphantom{xxxxxxxxxxxxxxxx}\, =(l_i-1)^2-l_i+1=0.
\end{eqnarray*}
 Therefore $f_i$ and $c_i$ are in the
orthogonal group of $b_i$.

Let $T=(R_1\times \dots \times R_s)\rtimes A$ be the semidirect
product of the trivial braces $R_1\times \dots\times R_s$ and $A$
via the action
$$\beta\colon A\longrightarrow \Aut(R_1\times\dots\times R_s)$$
defined by $\beta_{(a_1,\dots a_s)}(u_1,\dots
,u_s)=(c_1^{a_1}(u_1),\dots,c_s^{a_s}(u_s))$, where $\beta(a_1,\dots
,a_s)=\beta_{(a_1,\dots ,a_s)}$, for all $a_i\in \mathbb{Z}/(l_i)$
and all $u_i\in R_i$.

We define the map $\alpha\colon (\mathbb{Z}/(p),+)\longrightarrow
\Aut(T,+,\cdot)$  by
$$\alpha_{\mu}(u_1,\dots ,u_s,a_1,\dots
,a_s)=(f_1^{\mu}(u_1),\dots ,f_s^{\mu}(u_s),\gamma_1^{\mu}a_1,\dots
,\gamma_s^{\mu}a_s),$$ where $\alpha(\mu)=\alpha_{\mu}$, for all
$\mu\in\mathbb{F}_p$, all $a_i\in \mathbb{Z}/(l_i)$ and all $u_i\in
R_i$.

Let $b\colon T\times T\longrightarrow \mathbb{F}_p$ be the
symmetric bilinear form defined by
$$b((u_1,\dots ,u_s,a_1,\dots ,a_s),(v_1,\dots ,v_s,a'_1,\dots
,a'_s))=\sum_{i=1}^sb_i(u_i,v_i),$$ for all $a_i\in
\mathbb{Z}/(l_i)$ and all $u_i\in R_i$.

Since $c_i-\id$ is invertible, Theorem~\ref{main} yields that the
the asymmetric product $T\rtimes_{\circ} \mathbb{F}_p$ of $T$ by
$\mathbb{F}_p$, via $b$ and $\alpha$, is a simple  left brace.

 \begin{remark} {\rm In our approach we exploited the bilinear form with matrix
$B = (b_{ij})$, where $b_{ii} = 0$ and $b_{ij} = 1$ for $ i \neq j $.
Every permutation matrix is
in the orthogonal group of this bilinear form. In the concrete
constructions we show how to construct an element $f$ in the
orthogonal group that has the desired properties. However, one can
give a general reason why such elements exist.

Assume that $q$ and $p$ are prime numbers such that $p\mid(q-1)$.
Let $c\in S_{n}$ be an element of the symmetric group $S_{n}$ for
some $n$. Assume that $c$ has order $q$. Assume also that $\gamma
\in \mathbb Z_{q}^{*}$ has order $p$. We claim that
$c^{\gamma}=fcf^{-1}$ for a permutation $f\in S_{n}$ of order $p$.

Notice that $c^{\gamma}$ also has order $q$ (since $p,q$ are
relatively prime) and in fact $c,c^{\gamma}$ are products of the
same number of disjoint cycles of length $q$. It is then enough to
consider the case where $c=(1,2,\ldots, q)$, a cycle of length
$q$, and $n=q$. In particular, there exists  $g\in S_{q}$ such
that $gcg^{-1} =c^{\gamma}$. By the hypothesis, $\gamma^{p} =qr+1$
for some nonnegative integer $r$. Then
$$g^{p}cg^{-p}=c^{\gamma^{p}}=c^{qr+1}=c.$$
 It is easy to see that the centralizer $C_{S_{q}}(c)$ of $c$
in the symmetric group $S_{q}$ is equal to the cyclic group
$\langle c \rangle$. Therefore $g^{p}\in \langle c \rangle$. In
particular, $g^{pq}=1$. Since $p\mid (q-1)$, it follows that
$q=pk+1$ for some $k$ and then $g^{q}= g^{pk+1}\in g
C_{S_{q}}(c)$. Therefore, $g^{q}cg^{-q} =gcg^{-1} =c^{\gamma}$.

Notice that $g^{q}\neq 1$, as otherwise $c^{\gamma}=c$, a
contradiction, because $c$ has order $q$ and $\gamma \in \mathbb
Z_{q}^{*}$ has order $p$. Therefore, $g^{q}\neq 1$ and then
$g^{q}$ is a permutation of order $p$. Hence, we may put
$f=g^{q}$.}
\end{remark}

\section{Other simple left braces as asymmetric product of left
braces}\label{known}

Our aim in this section is to show that all previously known
nontrivial left simple braces (see \cite[Theorems~3.1
and~3.6]{BCJO}) can be interpreted as asymmetric products of braces.
This sheds a new light on these classes of examples and raises the
question of a further application of the asymmetric product
construction in the context of the challenging program of finding
new (simple) left braces.

Let $H(p^r,n,Q,f)$ be the left brace described in
\cite[Theorem~3.1]{BCJO}, where $p$ is a prime number,  $r,n$ are
positive integers, $Q$ is a quadratic form over
$(\mathbb{Z}/(p^r))^n$ (considered as a free module over the ring
$\mathbb{Z}/(p^r)$) and  $f$ is an element in the orthogonal group
of $Q$ of order $p^{r'}$ for some $0\leq r'\leq r$.

Recall that the additive group of $H(p^r,n,Q,f)$ is the additive
abelian group $(\mathbb{Z}/(p^r))^{n+1}$. The elements of
$(\mathbb{Z}/(p^r))^{n+1}$ will be written in the form
$(\vec{x},\mu)$, with $\vec{x}\in (\mathbb{Z}/(p^r))^n$ and $\mu\in
\mathbb{Z}/(p^r)$. The lambda map of $H(p^r,n,Q,f)$ is defined by
\begin{eqnarray}\label{lamb}
\lambda_{(\vec{x},\mu)}(\vec{y},\mu')=(f^{q(\vec{x},\mu)}(\vec{y}),
\mu'+b(\vec{x},f^{q(\vec{x},\mu)}(\vec{y}))),
\end{eqnarray}
for $(\vec{x},\mu),(\vec{y},\mu')\in (\mathbb{Z}/(p^r))^{n+1}$,
where $q(\vec{x},\mu)=\mu-Q(\vec{x})$, and $b$ is the bilinear form
$b(\vec{x},\vec{y})=Q(\vec{x}+\vec{y})-Q(\vec{x})-Q(\vec{y})$
associated to $Q$.

Let $T$ be the trivial brace $(\mathbb{Z}/(p^r))^n$ and let $S$ be
the trivial brace $\mathbb{Z}/(p^r)$. Let $\alpha\colon
(S,\cdot)\longrightarrow \Aut(T,+,\cdot)$ be the map defined by
$\alpha(\mu)=\alpha_{\mu}$ and
$$\alpha_{\mu}(\vec{x})=f^{\mu}(\vec{x}),$$
for all $\mu\in S$ and $\vec{x}\in T$. It is clear that $\alpha$ is
a homomorphism of groups. We will check that $\alpha$ and $-b$
satisfy condition (\ref{cond3}). Since $-b$ is a symmetric bilinear
form, this condition is equivalent to $-b(\vec{x},\vec{y})=
-b(\alpha_{\mu}(\vec{x}), \alpha_{\mu}(\vec{y}))$, which is true
because $f$ is in the orthogonal group of $Q$. Let
$H'(p^r,n,-b,f)=T\rtimes_{\circ}S$ be the asymmetric product of $T$
by $S$ via $-b$ and $\alpha$.

\begin{proposition}
The map $\varphi\colon H(p^r,n,Q,f)\longrightarrow H'(p^r,n,-b,f)$
defined by $\varphi(\vec{x},\mu)=(\vec{x},\mu-Q(\vec{x}))$ is an
isomorphism of left braces.
\end{proposition}

\begin{proof}
Let $\vec{x},\vec{y}\in T$ and $\mu,\mu'\in S$. We have
\begin{align*}
\varphi((\vec{x},\mu)+(\vec{y},\mu'))&=\varphi(\vec{x}+\vec{y},\mu+\mu')\\
&=(\vec{x}+\vec{y},\mu+\mu'-Q(\vec{x}+\vec{y}))\\
&=(\vec{x}+\vec{y},\mu+\mu'-b(\vec{x},\vec{y})-Q(\vec{x})-Q(\vec{y}))\\
&=(\vec{x},\mu-Q(\vec{x}))+(\vec{y},\mu'-Q(\vec{y}))\\
&=\varphi(\vec{x},\mu)+\varphi(\vec{y},\mu')
\end{align*}
and
\begin{align*}
\varphi((\vec{x},\mu)&\cdot(\vec{y},\mu'))\\
&=\varphi(\vec{x}+f^{\mu-Q(\vec{x})}(\vec{y}),\mu+\mu'+b(\vec{x},f^{\mu-Q(\vec{x})}(\vec{y})))\\
&=(\vec{x}+f^{\mu-Q(\vec{x})}(\vec{y}),\mu+\mu'+b(\vec{x},f^{\mu-Q(\vec{x})}(\vec{y}))-Q(\vec{x}+f^{\mu-Q(\vec{x})}(\vec{y})))\\
&=(\vec{x}+f^{\mu-Q(\vec{x})}(\vec{y}),\mu+\mu'-Q(\vec{x})-Q(f^{\mu-Q(\vec{x})}(\vec{y})))\\
&=(\vec{x}+f^{\mu-Q(\vec{x})}(\vec{y}),\mu+\mu'-Q(\vec{x})-Q(\vec{y}))\\
&=(\vec{x},\mu-Q(\vec{x}))\cdot (\vec{y},\mu'-Q(\vec{y}))\\
&=\varphi(\vec{x},\mu)\cdot\varphi(\vec{y},\mu').
\end{align*}
Thus the result follows.
\end{proof}

Consider now the left brace $H_1\bowtie\dots \bowtie H_s$ described
in \cite[Theorem~3.6]{BCJO}. Recall that
$H_i=H(p_i^{r_i},n_i,Q_i,f_i)$, $s>1$, $\;p_1,p_2,\ldots ,p_s$ are
different prime numbers, $p_1,p_2,\dots ,p_{s-1}$ are odd, $Q_i$ is
a non-degenerate quadratic form over $(\mathbb
Z/(p_i^{r_i}))^{n_i}$, $f_i$ is an element of  order $p_i^{r'_i}$ in
the orthogonal group determined by $Q_i$,  for some $0\leq r'_i\leq
r_i$, and the iterated matched product of left ideals is constructed
with the maps $$\alpha^{(j,i)}: (H_{j},\cdot) \longrightarrow \Aut
(H_i,+):(\vec{x}_{j},\mu_{j})\mapsto
\alpha^{(j,i)}_{(\vec{x}_{j},\mu_{j})},$$ with
$$\alpha^{(k+1,k)}_{(\vec{x}_{k+1},\mu_{k+1})}(\vec{x}_k,\mu_k)=(c_{k}^{\mu_{k+1}-Q_{k+1}(\vec{x}_{k+1})}(\vec{x}_k),
\mu_k),$$  for $1\leq k<s$,
$$\alpha^{(1,s)}_{(\vec{x}_1,\mu_1)}(\vec{x}_{s},\mu_{s})
=(c_{s}^{\mu_1-Q_{1}(\vec{x}_1)}(\vec{x}_{s}), \mu_{s}+
v_{s}((\id+c_{s}+\dots
+c_{s}^{\mu_1-Q_{1}(\vec{x}_1)-1})(\vec{x}_{s}))^t),$$
and
$\alpha^{(j,i)}_{(\vec{x}_{j},\mu_{j})}=\id_{H_i}$ otherwise, where,
for $1\leq i< s$, $c_{i}$ is an element of order $p_{i+1}^{r_{i+1}}$
in the orthogonal group determined by $Q_i$, $c_s$ is an element of
order $p_{1}^{r_{1}}$ of $\aut((\mathbb{Z}/(p_s^{r_s}))^{n_s})$ and
$v_{s}\in (\mathbb Z/(p_{s}^{r_{s}}))^{n_{s}}$, such that
\begin{eqnarray}\label{Qj+1}
&&Q_{s}(c_{s}(\vec{x}))=Q_{s}(\vec{x})+v_{s}\vec{x}^t,\end{eqnarray}
and $$f_ic_{i}=c_{i}f_i,$$ for  $1\leq i\leq s$. The lambda map of
$H_1\bowtie\dots \bowtie H_s$ is defined by
\begin{eqnarray*}
\lefteqn{\lambda_{(\vec{x}_{1},\mu_{1},\dots
,\vec{x}_{s},\mu_{s})}(\vec{y}_{1},\mu'_{1},\dots
,\vec{y}_{s},\mu'_{s})}\\
&=& (f_1^{\mu_1-Q_1(\vec{x}_1)}c_1^{\mu_2-Q_2(\vec{x}_2)}(\vec{y}_1),\mu'_1+b_1(\vec{x}_1,f_1^{\mu_1-Q_1(\vec{x}_1)}c_1^{\mu_2-Q_2(\vec{x}_2)}(\vec{y}_1)),\dots,\\
&&\qquad
f_{s-1}^{\mu_{s-1}-Q_{s-1}(\vec{x}_{s-1})}c_{s-1}^{\mu_s-Q_s(\vec{x}_s)}(\vec{y}_{s-1}),\mu'_{s-1}\\
&&\qquad +b_{s-1}(\vec{x}_{s-1},f_{s-1}^{\mu_{s-1}-Q_{s-1}(\vec{x}_{s-1})}c_{s-1}^{\mu_s-Q_s(\vec{x}_s)}(\vec{y}_{s-1})),\\
&&\qquad
f_s^{\mu_s-Q_s(\vec{x}_s)}c_s^{\mu_1-Q_1(\vec{x}_1)}(\vec{y}_s),\mu'_s+v_{s}((\id+c_{s}+\dots
+c_{s}^{\mu_1-Q_{1}(\vec{x}_1)-1})(\vec{y}_{s}))^t\\
&&\qquad
+b_s(\vec{x}_s,f_s^{\mu_s-Q_s(\vec{x}_s)}c_s^{\mu_1-Q_1(\vec{x}_1)}(\vec{y}_s)),
\end{eqnarray*}
for all $(\vec{x}_{i},\mu_{i}),(\vec{y}_{i},\mu'_{i})\in H_i$.

By \cite[Theorem~3.6]{BCJO} the left brace $H_1\bowtie\dots \bowtie
H_s$ is simple if and only if $c_i-\id$ is an automorphism for all
$1\leq i\leq s$.

Let $T'$ be the trivial brace
$(\mathbb{Z}/(p_1^{r_1}))^{n_1}\times\dots\times
(\mathbb{Z}/(p_s^{r_s}))^{n_s}$ and let $S'$ be the trivial brace
$\mathbb{Z}/(p_1^{r_1})\times\dots\times \mathbb{Z}/(p_s^{r_s})$.
Let $\alpha'\colon (S',\cdot)\longrightarrow \Aut(T',+,\cdot)$ be
the map defined by $\alpha'(\mu_1,\dots
,\mu_s)=\alpha'_{(\mu_1,\dots ,\mu_s)}$ and
$$\alpha'_{(\mu_1,\dots ,\mu_s)}(\vec{x}_1,\dots ,\vec{x}_s)=(f_1^{\mu_1}c_1^{\mu_2}(\vec{x}_1),\dots,f_{s-1}^{\mu_{s-1}}c_{s-1}^{\mu_s}(\vec{x}_{s-1}),f_s^{\mu_s}c_s^{\mu_1}(\vec{x}_s)),$$
for all $(\mu_1,\dots ,\mu_s)\in S'$ and $(\vec{x}_1,\dots
,\vec{x}_s)\in T'$. Let $b'\colon T'\times T'\longrightarrow S'$ be
the map defined by
$$b'((\vec{x}_1,\dots ,\vec{x}_s),(\vec{y}_1,\dots ,\vec{y}_s))=(-b_1(\vec{x}_1,\vec{y}_1),\dots ,-b_s(\vec{x}_s,\vec{y}_s)),$$
for all $(\vec{x}_1,\dots ,\vec{x}_s),(\vec{y}_1,\dots
,\vec{y}_s)\in T'$. We will check that $\alpha'$ and $b'$ satisfy
condition (\ref{cond3}). Since $b'$ is a symmetric bilinear form
this is equivalent to
$$b'((\vec{x}_1,\dots
,\vec{x}_s),(\vec{y}_1,\dots ,\vec{y}_s))= b'(\alpha'_{(\mu_1,\dots
,\mu_s)}(\vec{x}_1,\dots ,\vec{x}_s),\alpha'_{(\mu_1,\dots
,\mu_s)}(\vec{y}_1,\dots ,\vec{y}_s)),$$ which is true because, for
$1\leq i<s$, $c_i$ and $f_i$ are in the orthogonal group of $Q_i$,
$f_s$ is in the orthogonal group of $Q_s$ and
\begin{align*}
b_s(c_s(\vec{x}_s),c_s(\vec{y}_s))&=Q_s(c_s(\vec{x}_s)+c_s(\vec{y}_s))-Q_s(c_s(\vec{x}_s))-Q_s(c_s(\vec{y}_s))\\
&=Q_s(c_s(\vec{x}_s+\vec{y}_s))-Q_s(c_s(\vec{x}_s))-Q_s(c_s(\vec{y}_s))\\
&=Q_s(\vec{x}_s+\vec{y}_s)+v_s(\vec{x}_s+\vec{y}_s)^t-Q_s(\vec{x}_s)-v_s\vec{x}_s^t-Q_s(\vec{y}_s)-v_s\vec{y}_s^t\\
&=Q_s(\vec{x}_s+\vec{y}_s)-Q_s(\vec{x}_s)-Q_s(\vec{y}_s)\\
&=b_s(\vec{x}_s,\vec{y}_s).
\end{align*}
Let $B=T'\rtimes_{\circ}S'$ be the asymmetric product of $T'$ by
$S'$ via $b'$ and $\alpha'$.

\begin{theorem}
The map $\varphi'\colon H_1\bowtie\dots\bowtie H_s\longrightarrow B$
defined by
$$\varphi'(\vec{x}_1,\mu_1,\dots,\vec{x}_s,\mu_s)=(\vec{x}_1,\dots,\vec{x}_s,\mu_1-Q_1(\vec{x}_1),\dots
,\mu_s-Q_s(\vec{x}_s))$$ is an isomorphism of left braces.
\end{theorem}

\begin{proof}
 The proof is straightforward using the definitions of the sum and
the multiplication in both left braces and the formula (10).
\end{proof}

\section*{Acknowledgments}
The two first-named authors were partially supported by the grant
MINECO MTM2014-53644-P. The third author is supported in part by
Onderzoeksraad of Vrije Universiteit Brussel and Fonds voor
Wetenschappelijk Onderzoek (Belgium). The fourth author is supported
by a National Science Centre grant (Poland).

\vspace{30pt}
 \noindent \begin{tabular}{llllllll}
 D. Bachiller && F. Ced\'o  \\
 Departament de Matem\`atiques &&  Departament de Matem\`atiques \\
 Universitat Aut\`onoma de Barcelona &&  Universitat Aut\`onoma de Barcelona  \\
08193 Bellaterra (Barcelona), Spain    && 08193 Bellaterra (Barcelona), Spain \\
 dbachiller@mat.uab.cat &&  cedo@mat.uab.cat\\
   &&   \\
E. Jespers && J. Okni\'{n}ski  \\ Department of Mathematics &&
Institute of Mathematics
\\  Vrije Universiteit Brussel && Warsaw University \\
Pleinlaan 2, 1050 Brussel, Belgium && Banacha 2, 02-097 Warsaw, Poland\\
efjesper@vub.ac.be&& okninski@mimuw.edu.pl
\end{tabular}


\begin{thebibliography}{99}
\itemsep=-2pt
\bibitem{B} D. Bachiller, Classification of braces of order $p^3$, J.
Pure Appl. Algebra 219 (2015), 3568--3603.
\bibitem{B2} D. Bachiller, Counterexample to a conjecture about
braces, J. Algebra 453 (2016), 160--176.
\bibitem{B3} D. Bachiller, Extensions, matched products, and simple braces,
arXiv:1511.08477v3 [math.GR].
\bibitem{BCJ} D. Bachiller, F. Ced\'o and E. Jespers, Solutions of
the Yang-Baxter equation associated with a left brace,  J. Algebra
463 (2016), 80--102.
\bibitem{BCJO1} D. Bachiller, F. Ced\'o, E. Jespers and J.
Okni\'{n}ski, A family of irretractable square-free solutions of the
Yang-Baxter equation, Forum Math. (2017), DOI:
10.1515/forum-2015-0240, arXiv:1511.07769v2 [math.GR].
\bibitem{BCJO} D. Bachiller, F. Ced\'o, E. Jespers and J.
Okni\'{n}ski, Iterated matched products of finite braces and
simplicity; new solutions of the Yang-Baxter equation, Trans. Amer.
Math. Soc, to appear,  arXiv:1610.00477v1[math.GR].
\bibitem{BDG2} N. Ben David and Y. Ginosar, On groups of $I$-type and
involutive Yang-Baxter groups,  J. Algebra 458 (2016), 197--206.
\bibitem{CCS} F. Catino, I. Colazzo and P. Stefanelli, Regular
subgroups of the afine group and asymmetric product of braces, J.
Algebra 455 (2016), 164--182.
\bibitem{CGIS} F. Ced\'o, T.
Gateva-Ivanova and A. Smoktunowicz, On the Yang-Baxter equation and
left nilpotent left braces, J. Pure Appl. Algebra 221 (2017),
751--756.
\bibitem{CJO} F. Ced\'o, E. Jespers and J. Okni\'{n}ski,
Retractability  of the set theoretic solutions of the Yang-Baxter
equation, Adv. Math. 224 (2010), 2472--2484.
\bibitem{CJO2} F. Ced\'o, E. Jespers
and J. Okni\'{n}ski, Braces and the
 Yang-Baxter equation, Commun. Math. Phys. 327 (2014), 101--116.
\bibitem{CJR} F. Ced\'o, E. Jespers and \'{A}. del R\'{\i}o, Involutive
Yang-Baxter groups, Trans. Amer. Math. Soc. 362 (2010), 2541--2558.
\bibitem{ESS} P. Etingof, T. Schedler, A.  Soloviev,  Set-theoretical solutions
 to the quantum Yang-Baxter equation, Duke Math. J. 100 (1999), 169--209.
\bibitem{GI} T. Gateva-Ivanova, A combinatorial approach to the
set-theoretic solutions of the Yang-Baxter equation, J. Math. Phys.
45 (2004), 3828--3858.
\bibitem{GIC} T. Gateva-Ivanova and P. Cameron, Multipermutation
solutions of the Yang-Baxter equation, Comm. Math. Phys. 309 (2012),
583--621.
\bibitem{gat-maj} T. Gateva-Ivanova and S. Majid,
Matched pairs approach to set theoretic solutions of the Yang-Baxter
equation, J. Algebra 319 (2008), no. 4, 1462--1529.
\bibitem{GIVdB} T. Gateva-Ivanova and M. Van den Bergh,
  Semigroups of $I$-type, J. Algebra 206 (1998), 97--112.
\bibitem{GI-braces} T. Gateva-Ivanova, Set-theoretic solutions of the
Yang-Baxter equation, braces and symmetric groups, preprint arXiv:
1507.02602v4 [math.QA].
\bibitem{gor} D. Gorenstein, Finite groups, second edition, Chelsea Publishing Company, New York,1980.
\bibitem{Huppert} B. Huppert, Endliche Gruppen I, Springer-Verlag, Berlin 1967.
\bibitem{K} C. Kassel, Quantum Groups, Springer-Verlag, 1995.
\bibitem{Rump1} W. Rump, A decomposition theorem for square-free
unitary solutions of the quantum Yang-Baxter equation, Adv. Math.
193 (2005), 40--55.
\bibitem{Rump} W. Rump, Braces, radical rings, and the quantum Yang-Baxter
 equation, J. Algebra 307 (2007), 153--170.
\bibitem{Smokt} A. Smoktunowicz, A note on set-theoretic solutions
of the Yang-Baxter equation, J. Algebra (2016),
http://dx.doi.org/10.1016/j.jalgebra.2016.04.015
\bibitem{Smokt2} A. Smoktunowicz, On Engel groups, nilpotent groups, rings, braces and the Yang-Baxter
equation, Trans. Amer. Math. Soc., to appear. DOI:
https://doi.org/10.1090/tran/7179.
\bibitem{ven2016}  L. Vendramin,  Extensions of set-theoretic solutions of the Yang-Baxter equation and a conjecture of Gateva-Ivanova, J. Pure Appl. Algebra 220 (2016), no. 5, 2064--2076.
\bibitem{Yang} C. N. Yang, Some exact results for the many-body
problem in one dimension with repulsive delta-function interaction,
Phys. Rev. Lett. 19 (1967), 1312--1315.
\end{thebibliography}
\end{document}